\documentclass[11pt, a4paper, reqno]{amsart}
\usepackage{mathrsfs}
\usepackage{amsfonts}
\usepackage{amsmath}
\usepackage{graphicx}
\usepackage{amssymb}
\usepackage{amsthm}
\usepackage{color}
\usepackage{hyperref}
\usepackage{vmargin}

\long\def\symbolfootnote[#1]#2{\begingroup%
\def\thefootnote{\fnsymbol{footnote}}\footnote[#1]{#2}\endgroup}

\setmarginsrb{20mm}{20mm}{20mm}{20mm}{10mm}{10mm}{10mm}{10mm}

\newtheoremstyle{remark}
  {}{}{}{}{\bfseries}{.}{.5em}{{\thmname{#1 }}{\thmnumber{#2}}{\thmnote{ (#3)}}}

\DeclareMathOperator\dist{dist}

\RequirePackage{amsthm}
\newtheorem*{theorem*}{Theorem}
\newtheorem{theorem}{Theorem}[section]
\newtheorem{lem}[theorem]{Lemma}

\newtheorem{prop}[theorem]{Proposition}

\theoremstyle{definition}

\theoremstyle{remark}
\newtheorem{rem}[theorem]{Remark}
\newtheorem{ex}{Example}

\def\div{{\rm div}}
\def\vint{\mathop{\mathchoice%
          {\setbox0\hbox{$\displaystyle\intop$}\kern 0.22\wd0%
           \vcenter{\hrule width 0.6\wd0}\kern -0.82\wd0}%
          {\setbox0\hbox{$\textstyle\intop$}\kern 0.2\wd0%
           \vcenter{\hrule width 0.6\wd0}\kern -0.8\wd0}%
          {\setbox0\hbox{$\scriptstyle\intop$}\kern 0.2\wd0%
           \vcenter{\hrule width 0.6\wd0}\kern -0.8\wd0}%
          {\setbox0\hbox{$\scriptscriptstyle\intop$}\kern 0.2\wd0%
           \vcenter{\hrule width 0.6\wd0}\kern -0.8\wd0}}%
          \mathopen{}\int}

\newcommand{\Om}{\Omega}

\newcommand{\R}{\mathbb{R}}

\newcommand{\ep}{\epsilon}

\newcommand{\Ric}{{\rm Ric}}
\newcommand{\Hess}{{\rm Hess}}

\newcommand{\supp}{{\rm supp}}
\newcommand{\vp}{\varphi}

\newcommand{\ud}{\mathrm {d}}
\def\barF{{\overline {F}}}
\def\barG{{\overline {G}}}

\def\bz{{\overline {z}}}

\newcommand{\tr}{\operatorname{tr}}

\newcommand\de{\delta}

\newcommand\cA{\mathcal{A}}   
\newcommand\fra{\mathfrak{a}}  
   
\newcommand\frb{\mathfrak{b}}

   \newcommand\frk{\mathfrak{k}} 
    
   \newcommand\frm{\mathfrak{m}}

   \newcommand\frp{\mathfrak{p}}

\definecolor{blau}{rgb}{0.1,0.0,0.9}
\definecolor{violet}{rgb}{0.54, 0.17, 0.89}
\newcommand{\blue}{\color{blau}}

\newcommand{\kom}[1]{}
\renewcommand{\kom}[1]{{\bf \blue /#1/}}

\newcounter{komcounter}
\numberwithin{komcounter}{section}

\title[]{Isoperimetric inequalities and regularity of $A$-harmonic functions on surfaces
}

\author[]{Tomasz Adamowicz{\small$^1$}}
\address{T.A.: Institute of Mathematics, Polish Academy of Sciences,
\'Sniadeckich 8, Warsaw, 00-656, Poland\/}
\email{tadamowi@impan.pl}

\author[]{Giona Veronelli}
\address{G.V.: Dipartimento di Matematica e Applicazioni, Università di Milano Bicocca, via R. Cozzi 55, 20126 Milano, Italy}
\email{giona.veronelli@unimib.it}

\begin{document}


\footnotetext[1]{T. Adamowicz was supported by a grant of National Science Center, Poland (NCN),
 UMO-2017/25/B/ST1/01955.}

\begin{abstract}
We investigate the logarithmic and power-type convexity of the length of the level curves for $a$-harmonic functions on smooth surfaces and related isoperimetric inequalities. In particular, our analysis covers the $p$-harmonic and the minimal surface equations. As an auxiliary result, we obtain higher Sobolev regularity properties of the solutions, including the $W^{2,2}$ regularity. 
 
   The results are complemented by a number of estimates for the derivatives $L'$ and $L''$ of the length of the level curve function $L$, as well as by examples illustrating the presentation.
 
  Our work generalizes results due to Alessandrini, Longinetti, Talenti and Lewis in the Euclidean setting, as well as a recent article of ours devoted to the harmonic case on surfaces.
\end{abstract}

\maketitle

\section{Introduction}

The main goal of this work is to continue studies of the geometry of level curves for functions on smooth surfaces initiated in~\cite{adve} in the setting of harmonic functions. Namely, we expand the scope of the studied PDEs to include $a$-harmonic equations on surfaces under relatively mild assumptions on the operator $a$, see the presentation in this section below. Our results apply in particular to the $p$-harmonic equation. However, we do not impose the assumption of the homogeneity of degree $p$ of the operator, i.e. $a$ in \eqref{def-strong-aharm} is homogeneous,
 typically required in the $p$-harmonic setting. This equation is one of the key nonlinear counterparts of the harmonic one ($p=2$), and is sometimes called "a mascot of nonlinear analysis" (P. Drabek). The $p$-harmonic equation appears, for instance, in non-newtonian fluid dynamics, the description of the Hele--Shaw flow, the image processing and stochastic games, see~\cite[Chapter 2]{lund} and references therein.  The $p$-harmonic equation plays a key role in potential analysis~\cite{hkm} and in geometric analysis~\cite{holop,ps}, especially in the two-dimensional setting~\cite{aim, arli, man}. Our results extend a work by Alessandrini~\cite{al2} for the flat case, who investigated the convexity of the logarithm (or powers) of the length of level sets for $a$-harmonic type functions on planar annuli with constant boundary data, see also Laurence~\cite{la} and Longinetti~\cite{long} for harmonic and $p$-harmonic functions in $\R^2$. Furthermore, our $a$-harmonic variant of the no-critical points lemma presented in  Lemma~\ref{non-zero-lemma} extends Lewis's result ~\cite{lew3} for $p$-harmonic functions on $\R^2$ to the setting of surfaces. 

Another important equation investigated in this note is the minimal surface one. Such equation is one of the most fundamental geometric PDEs with connections to the Plateau problem, harmonic mappings theory and the geometry and topology of manifolds, see e.g~\cite{cm-book}. Moreover, the minimal surface equation has also been vividly studied on Riemannian surfaces, including the case of surfaces of nonpositive curvature we are more interested in. A nonexhaustive list of studied topics in this direction includes, for instance, the isoperimetric estimates on minimal graphs
\cite{bm} (in any dimension), the existence and properties of entire minimal graphs \cite{chr,km,gr} and of special minimal surfaces \cite{m}, as well as the special cases of minimal graphs over the hyperbolic plane, see e.g. \cite{cr,hnset,km,st}. Our results for minimal graphs extend previous ones by Longinetti~\cite{long-jde}, also in the Euclidean setting. Let us also mention that other well-known equations studied in the literature are covered by our work, for example the subsonic gas equation and the maximal graph equation in Lorentzian spacetime; see below.

We will now discuss the setting of $a$-harmonic equations and present the main results together with the organization of the paper. Consider a function
$a \in C^1(0,\infty)$ with $a(s)>0$ for all $s>0$ and such that, for some constants $0<\alpha\leq \beta$, it holds that

\begin{equation}\tag{A}
\qquad  0<\alpha \leq 1+\frac{a'(s)s}{a(s)}\leq \beta,\quad \hbox{for all } s>0.  \label{a-harm-ass1}
\end{equation}
Moreover, additionally we require that
\begin{equation}\label{a-harm-ass2}\tag{A'}
 sa(s)\to 0 \quad \hbox{ for }s\to 0.
\end{equation}
The latter assumption corresponds to the second part of condition (A2) in~\cite{puse-book}, one of the general assumptions imposed on the $a$-harmonic operators there. Observe that its first part, i.e. that function $sa(s)$ is strictly increasing, follows from~\eqref{a-harm-ass1} by differentiation, as $(sa(s))'\geq \alpha a(s)>0$ for $s>0$. Condition~\eqref{a-harm-ass2} is used in several results below, for instance in the proof of the key result, Theorem~\ref{thm-main1}, where we appeal to the strong maximum principle for equation~\eqref{def-weak-aharm}, see~\cite[Theorem 8.5.1]{puse-book}. Furthermore, we need~\eqref{a-harm-ass2} in the higher regularity result stated as Proposition~\ref{lem-reg}, which in turn enters in the study of critical points in Section~\ref{Section3.1}. 

Let $(M^n, g)$ be a Riemannian manifold and $\Om\subset M^n$ be an open set. In this paper, we will focus mainly (but not only) on surfaces, i.e. $n=2$. In what follows we denote the Riemannian norm of the gradient of a function $u$ by $|\nabla u|_g$ and use also $|\nabla u|$ if the metric $g$ is fixed or clearly understood from the context of presentation. We also denote by $|\nabla u|_0$ the Euclidean norm of the Euclidean gradient $u$. Namely, in a given coordinate chart $|\nabla u|^2_g:=g^{ij}\partial_iu\partial_ju$ and $|\nabla u|_0^2:=\delta^{ij}\partial_iu\partial_ju=\sum_i (\partial_iu)^2$.

Following the approach in~\cite[Chapter 3.1]{puse-book}, we say that a weakly differentiable function $u\in L^1_{loc}(\Om, \R)$ is \emph{$a$-harmonic} in $\Om$, if $a(|\nabla u(\cdot)|_g)\nabla u(\cdot)\in L^1_{loc}(\Om)$ and if $u$ is a weak solution of the equation 
\begin{equation}\label{def-strong-aharm}
\div(a(|\nabla u|_g)\nabla u)=0\quad \hbox{in }\Om,
\end{equation}
meaning that
\begin{equation}\label{def-weak-aharm}
 \int_{\Om} a(|\nabla u|_g) \langle \nabla u, \nabla \phi\rangle_g \ud M=0
\end{equation}
holds for all compactly supported test functions $\phi\in C_{0}^1(\Om)$. In what follows, we will call the divergence-type operator in~\eqref{def-strong-aharm} the \emph{$a$-harmonic operator}.

Before stating our main results, let us present some examples of operators $a$ satisfying the assumptions \eqref{a-harm-ass1} and \eqref{a-harm-ass2}.
 
\begin{ex}[The $p$-harmonic equation] The most important example of an $a$-harmonic function is the $p$-harmonic one for $a(s):=s^{p-2}$.  Condition~\eqref{a-harm-ass1} is satisfied with $\alpha=\beta:=p-1$, and since $sa(s)=s^{p-1}$, condition~\eqref{a-harm-ass2} holds as well. For a bounded domain $\Om$, a function $u\in W^{1,p}_{loc}(\Om, \R)$ satisfying ~\eqref{def-weak-aharm} with this choice of $a$ is called $p$-harmonic. For this choice of $a$, it is enough to require~\eqref{def-weak-aharm} for all compactly supported test functions $\phi\in W^{1,p}(\Om)$.
\end{ex}

\begin{ex}[The minimal surface equation] By choosing $a(s):=1/\sqrt{1+s^2}$ we retrieve the \emph{minimal surface equation}, see~\cite{al2, long-jde}. In this case condition~\eqref{a-harm-ass2} holds by a direct computation. Moreover, we have 
\begin{equation*}
0\le 1+\frac{a'(s)s}{a(s)}=(1+s^2)^{-1}\leq 1=\beta,\quad \hbox{for all } s>0.  
\end{equation*}
If one disposes of an upper bound $|\nabla u|<C$ on the gradient of a solution, then $a$ satisfies \eqref{a-harm-ass1} with $\alpha=\inf \frac{1}{1+C^2}$. In particular, this is the case if a $C^1$ solution is a priori given, as in the statement of Theorem \ref{thm-main2}.
\end{ex}

\begin{ex}[The subsonic gas equation] Basing on the discussion on pg. 262 in~\cite{gt} we may point to yet another quasilinear equation covered by our theory, here in $\R^2$. Namely, the potential gas flow equation $\div (a(|\nabla u|)\nabla u)=0$, where $u$ stands for the velocity potential of the flow and $a(|\nabla u|)$ expresses the fluid density-speed relation. Here, 
	\[
	a(s):=\left(1-\frac{\gamma-1}{2}s^2\right)^{\frac{1}{\gamma-1}},
	\]
	where $\gamma>1$ is constant and the equation is elliptic if the flow is subsonic, which means that
	$|\nabla u|^2<\frac{2}{\gamma+1}$. Indeed, under this assumption $a>\frac{2}{\gamma+1}$ and $a\in C^1(0,\infty)$. Let us verify assumptions~\eqref{a-harm-ass1} and \eqref{a-harm-ass2} on $a$. By direct computations we check that
	\[
	1+\frac{a'(s)s}{a(s)}=1-\frac{s^2}{1-\frac{\gamma-1}{2}s^2}=\frac{1-\frac{\gamma+1}{2}s^2}{1-\frac{\gamma-1}{2}s^2}.
	\]
	The latter expression takes value $1$ at $s=0$, and value $\frac{\gamma^2-1}{\gamma^2+3}$ when $s\to \frac{2}{\gamma+1}$, hence by taking $\alpha:=\frac{\gamma^2-1}{\gamma^2+3}$ and $\beta:=1$ we obtain that condition~\eqref{a-harm-ass1} holds for this $a$-harmonic equation. Moreover, by the direct argument, $sa(s)\to 0$, as $s\to 0^{+}$, giving us~\eqref{a-harm-ass2}.  
\end{ex}

\begin{ex}[The maximal graph equation in Lorentzian spacetime]\label{ex-max-Lor}$\phantom{1}$ The maximal graph equation corresponds to the choice $a(s):=1/\sqrt{1-s^2}$. Maximal surfaces in a Lorentzian manifold are spacelike surfaces with zero mean curvature. In the Lorentz-Minkowski space $\mathbb L^3$ they arise as
	local maxima for the area functional associated to variations of the surface
	by spacelike surfaces; see e.g. \cite{mati} for the physical motivations, \cite{ca,chya} for some celebrated results on maximal graphs, and \cite{al,alal} for some more recent works on surfaces. Moreover, the maximal graph equation is closely related to the Born-Infeld equation, see~\cite{bdp}. In the case of strictly spacelike solutions, $u\in C^1$ and $|\nabla u|<1$, see~\cite{basi}. In particular, given a strictly spacelike solution on a compact set, it holds that $|\nabla u|\le C$ for some constant $C<1$, so that $a(s)$ satisfies condition~\eqref{a-harm-ass1} with $\alpha=1$ and $\beta=\sup \frac{1}{1-C^2}$. Furthermore, the direct computations give us~\eqref{a-harm-ass2}.   
\end{ex}

 Let us remark that our definition of $a$-harmonic equation is a special case of a  slightly more general class of PDEs, the so called $A$-harmonic PDEs, where $A:=A(z,\nabla u)\nabla u$, and
\[
 \div(A(z, \nabla u)\nabla u)=0,
\]
see the dedicated monograph~\cite{hkm}. In particular, we retrieve ~\eqref{def-weak-aharm} in~\cite{hkm}  by setting $A:=a(|\nabla u|)\nabla u$.

In all the main examples of $a$-harmonic functions we mentioned above,  $\log a(s)$ is trivially either lower or upper bounded in a neighborhood of $0$. However, in all generality this property is not trivially implied by \eqref{a-harm-ass1} and \eqref{a-harm-ass2}; see Example \ref{ex_valtorta} in Section 3 below. 
Therefore, in order to obtain the regularity estimates in Proposition \ref{lem-reg} we need to impose an additional technical assumption on operators $a$, requiring that 
\begin{equation}\label{a-harm-ass3} \tag{A''}
\hbox{ $\log a(s)$ is either upper or lower bounded (or both) on $(0,1]$. }
\end{equation}

 The key boundary value problem studied in this work is as follows. Let $t_1,t_2\in \R$ be such that $t_1<t_2$ and let us consider a continuous up to the boundary solution of the following Dirichlet problem in an annulus $\Om$ for the $a$-harmonic operator on a two dimensional Riemannian manifold $M^2$: \begin{equation}\label{DP}\tag{DP}
 \begin{cases}
\div(a(|\nabla u|_g)\nabla u)=0\,\, \hbox{in } \Om,\\
 u|_{\Gamma_1}=t_1,\quad u|_{\Gamma_2}=t_2.
 \end{cases}
 \end{equation}
From here on, by (topological) annulus we mean a domain $\Omega\Subset M^2$ homeomorphic to the flat concentric annulus $\{x\in\R^2\,:\,1<\|x\|<R\}\subset \R^2$ for some $R>1$. Note that, up to possibly modifying the value of $R$, a non-denegenerate topological annulus is actually conformal to a flat annulus by a version of the uniformization theorem, see \cite[Lemma 3.6]{adve} specialized to smooth surfaces.
Then, $\Gamma_1$ and $\Gamma_2$ stand for the  $C^{1,\alpha}$ connected boundary components of $\Om$. The main class of examples is obtained for $\Omega:=\Omega_1\setminus \overline\Omega_2$ with $\Omega_2\Subset\Omega_1$ two domains of $M^2$ homeomorphic to a ball. However, in general the topology of $\Omega_2$, and hence of $\Omega_1$, could be nontrivial. Since now on, we will denote by $K$ the Gaussian curvature of $M^2$. It is worthy recalling that on smooth surfaces $\Ric(\nabla u, \nabla u)\equiv K|\nabla u|^2$, where $\Ric$ is the Ricci curvature tensor. Let us also add that the Dirichlet problem analogous to~\eqref{DP} can be considered on $n$-dimensional manifolds. It is our future project to extend results of this work beyond the setting of surfaces.\
 
We are now in a position to present the main results of the paper, proven in {\bf Section 2}. The following theorem gives a counterpart of Alessandrini's result~\cite[Theorem 1.1]{al2} for $a$-harmonic functions on non-positively curved Riemannian $2$-manifolds and extends our previous work in~\cite[Theorems 2.7]{adve} devoted to the harmonic case. Recall that we do not assume that the curvature is constant, i.e. we present the argument for $K=K(x)$ for $x\in M^2$. Notice that, following \cite{al2,la,long}, the second part of the assertion motivates the name isoperimetric inequality. 

In what follows by $L(t)$ we denote the length (i.e.  the $1$-dimensional Hausdorff measure) of the level set $\{u=t\}$, cf.~\eqref{iso-in-length} below. Moreover, by $L'$ and $L''$ we denote, respectively, the first and the second derivative of $L$ with respect to parameter $t$. 

\begin{theorem}[Isoperimetric inequality for $a$-harmonic equations]\label{thm-main1}
	
Let $\Omega$ be a $C^{1,\alpha}$-topological annular domain in a $2$-di\-men\-sional Riemannian manifold $(M^2, g)$ of non-positive curvature $K|_\Om \leq 0$. Let $t_1,t_2\in \R$ be such that $t_1<t_2$ and let us consider a continuous up to the boundary $a$-harmonic solution $u$ of the Dirichlet problem~\eqref{DP} in $\Om$ satisfying conditions~\eqref{a-harm-ass1}, \eqref{a-harm-ass2} and ~\eqref{a-harm-ass3}. Then  
\begin{align}
&(\ln L(t))''\geq 0\quad\hbox{for all } t\in (t_1,t_2),\hbox{ if $\beta=1$}\label{e:logconv}\\
&\left(\frac{1}{m} L^m(t)\right)''\geq 0\quad\hbox{for all } t\in (t_1,t_2),\hbox{ if $\beta\neq 1$ and $m=\frac{\beta-1}{\beta}$},\label{e:logconv-p}
\end{align}
where $\beta$ is as in \eqref{a-harm-ass1}. Moreover, the equality in the assertion holds on $(t_1,t_2)$ if and only if $u$ is $(\beta+1)$-harmonic and $K\equiv 0$, in which case $\Om$ is locally isometric to the regular (circular) annulus in the plane and
	all the level curves of $u$ are locally concentric circles. 
\end{theorem}

We present now our second main result. As observed above, a special case of the $a$-harmonic equation is the minimal surface equation
\begin{equation}\label{eq-min-surf}
 \div\left( \frac{1}{\sqrt{1+|\nabla u|^2}}\nabla u\right)=0,
\end{equation}
for which $a(s)=1/\sqrt{1+s^2}$ and $\alpha<1/(1+\sup_{\Om}|\nabla u|^2)$ and $\beta=1$. Theorem \ref{thm-main1} applied to \eqref{eq-min-surf} is not sharp, since
$1+a'(s)s/a(s)<1$. However, in the setting of convex planar annuli, using the support functions of a convex set, Longinetti in~\cite[Theorem 2.2]{long-jde} showed the following sharp isoperimetric inequality
\[
 LL''-(L')^2\geq 4\pi^2.
\]
It turns out that our approach allows to extend Longinetti's result to non-necessarily convex annuli in non-positively curved smooth surfaces, thus weakening the assumptions also in the Euclidean case.

More precisely, let us consider a class of functions $a\in C^1(0,\infty)$ which satisfy the following growth condition
\begin{equation}\label{a-harm-ass-min}
0\leq \alpha\leq 1+\frac{a'(s)s}{a(s)}\leq \frac{1}{1+s^2}.
\end{equation}
On one hand this condition restricts assumption ~\eqref{a-harm-ass1} by imposing the sharper upper growth, but on the other hand we allow now $\alpha$ to be zero, as implied by the case of unbounded gradient, i.e. $s\to \infty$. 

Notice that in the case of the minimal surfaces equation the equality holds in the right hand side of \eqref{a-harm-ass-min}.

\begin{theorem}[Isoperimetric inequality for minimal surfaces equation]\label{thm-main2}
 Let $\Omega=\Omega_1\setminus\Omega_2$ be a $C^{1,\alpha}$-topological annular domain in a $2$-di\-men\-sional Riemannian manifold $(M^2, g)$ with $\Om_1$ and $\Om_2$ homeomorphic to balls, and suppose that $K|_{\Om_1} \leq 0$. Let further $t_1,t_2\in \R$ be such that $t_1<t_2$ and let us consider a continuous up to the boundary $a$-harmonic $C^2(\Omega)$ solution $u$ of the following Dirichlet problem:
 \begin{equation}\label{DP-min}
 \begin{cases}
 \div\left(a(|\nabla u|_g)\nabla u\right)=0 & \hbox{in } \Om,\\
 u|_{\Gamma_1}=t_1, \quad u|_{\Gamma_2}=t_2,
 \end{cases}
\end{equation}
with $a$ satisfying assumptions~\eqref{a-harm-ass-min}, \eqref{a-harm-ass2} and \eqref{a-harm-ass3}.

Then  
\begin{align}\label{assert-thm-main2}
 L(t)L''(t)-(L'(t))^2\geq 4\pi^2, \quad \hbox{for all }t\in (t_1,t_2).
\end{align}

Moreover, the equality in the assertion holds on $(t_1,t_2)$ if and only if $K\equiv 0$ and $u$ is the solution to the Dirichlet problem \eqref{DP-min} for the minimal surface equation with $a(s)=1/\sqrt{1+s^2}$. In such a case
$\Om$ is locally isometric to the regular (circular) annulus in the plane and all the level curves of $u$ are locally concentric circles.
\end{theorem} 

\begin{rem}\label{rem-27}
	Our assumption that in the definition of $\Omega=\Om_2\setminus \overline{\Om_1}$ both components $\Om_1$ and $\Om_2$ are homeomorphic to balls, is crucial to obtain the constant $4\pi^2$ above. Indeed, otherwise this constant changes according to the topology of component $\overline{\Om_1}$, as the Gauss-Bonnet theorem is invoked in the proof. For example, if $\Om$ is the topologically non-trivial annulus $[0,1]\times \mathbb S^1$ in $\R\times \mathbb{S}^1$, then $u(t,\theta)=t$ solves the minimal graph equation, but $LL''-(L')^2\equiv 0$ in this case. 
\end{rem}
Unlike Theorem~\ref{thm-main1}, in the above result we a priori require solutions to be $C^2$-regular. This is due to the fact that in Theorem~\ref{thm-main2} the assumption~\eqref{a-harm-ass1} is replaced by~\eqref{a-harm-ass-min}. In the proof of Theorem~\ref{thm-main1}, condition ~\eqref{a-harm-ass1} is used to prove the $W^{1,2}_{loc}$ regularity of $a^{1/2}(|\nabla u|_g)\nabla u$ (see computations in Section~\ref{Section3.1}) and, hence, to prove in Lemma~\ref{non-zero-lemma} that there are no critical points of solutions (see also the discussion in the beginning of the proof of Theorem~\ref{thm-main2}). On the other hand, the $C^2$-regularity assumption is not much restrictive, as it is satisfied in several cases where the solution exists. Indeed, let us comment on existence of solutions to the Dirichlet problem~\eqref{DP-min}. The well-known phenomenon, discussed e.g. in~\cite{sw, sim2}, shows that for the minimal surface equation operator $a(s)=1/\sqrt{1+s^2}$, the $C^2$-solution to the Dirichlet problem exists under a smallness assumption of the boundary data, see \cite[Theorem 2]{sw}, which in our case imposes the bound on $|t_2-t_1|<c$. Otherwise, the solution may fail to exist, see e.g.  Section IV in~\cite{sw}. Moreover, examples in~\cite{sw} and \cite{sim1}, provide us with wide classes of $a$-harmonic equations and boundary values $t_1,t_2$ with solvable ~\eqref{DP-min}, see ~\cite[Theorem 3]{sw} and Sections 4 and 5 in~\cite{sim1}. 

\begin{rem}
An isoperimetric inequality similar to \eqref{assert-thm-main2} can be obtained as well for maximal surfaces in Lorentzian spacetime, cf. Example~\ref{ex-max-Lor}. Namely, it holds
\begin{equation}\label{isop-ineq-Lorentz}
 \left(L'(t)+\int_{\{x\in \Om\,:\,u(x)=t\}} k\right)^2 \le L(t)L''(t),
\end{equation}
where $k$ is the curvature of the level curve with respect to the outward normal direction.
In particular, if $M^2=\R^2$ then $\int_{\{x\in \Om\,:\,u(x)=t\}} k
=2\pi$, so that we have the following relation
\[
\left( L'(t)+2\pi\right)^2 \le L(t)L''(t).
\]
See the comment after the proof of Theorem \ref{thm-main2} for the derivation of~\eqref{isop-ineq-Lorentz}. 
\end{rem}

 {\bf Section 3} contains some key auxiliary results about the regularity of solutions to our Dirichlet problem~\eqref{DP} and the no-critical points lemma. Such results, although known in the flat case, are new in the setting of surfaces.  In particular, Proposition~\ref{lem-reg} shows the $W^{2,2}$-regularity of $a$-harmonic functions on open sets on surfaces (no curvature bounds are assumed here) and the $W^{1,1}$-Sobolev regularity of the differential expression $a^{\delta}(|\nabla u|)\nabla u$ for $\delta \in  (0,1]$. The novelties here are the following. First, on the contrary to the difference quotient method used for obtaining such result in the plane or for $p$-harmonic functions, see e.g.~\cite[Section 16]{aim} and \cite{strzelecki}, we use a perturbation method which does not require the function $a$ to be homogeneous.  Furthermore, Proposition~\ref{lem-reg} corresponds to~\cite[Theorem 2.1]{cima} and our equation~\eqref{def-strong-aharm} is slightly different than \cite[(2.1)]{cima}, while conditions (2.2)-(2.3) in~\cite{cima} correspond to our condition~\eqref{a-harm-ass1}. Additionally, in order to show Proposition~\ref{lem-reg}, we extend the stream function method to the setting of surfaces. Moreover, in Theorem~\ref{th:regularity} we adapt Talenti's general result~\cite{ta} to the Riemannian setting writing a simple and short proof of the $W^{2,2}$-regularity of $a$-harmonic functions on surfaces. Such a result is to our best knowledge not stated explicitly in the literature.

The regularity proven in Proposition~\ref{lem-reg} is employed in {\bf Section 3.1}, where we discuss a complex elliptic system of equations satisfied by the complex gradient of the $a$-harmonic solution on a surface, see \eqref{eq-final-cmplx} and~\eqref{eliptic-cond}. In that context, such systems have been so far studied only for the $p$-harmonic equation. The key consequence is the complex representation formula~\eqref{F-sol} which allows us to infer information about the structure of critical points of the $a$-harmonic solution. As a byproduct we observe the unique continuation property for $a$-harmonic functions on smooth surfaces, which was not stated explicitly in the literature even for the $p$-harmonic equation, see Proposition~\ref{prop:uni-cont}. Both Proposition~\ref{lem-reg} and the complex representation of solutions on surfaces serve in proving Lemma~\ref{non-zero-lemma}. This latter generalizes the harmonic result in~\cite[Lemma 2.9]{adve} to the $a$-harmonic setting and, moreover, Lewis's $p$-harmonic result~\cite{lew3} to the setting of surfaces.

Let us emphasize that the presence of the curvature, which may vary from point to point, leads to the form of equation~\eqref{def-strong-aharm} with conformal factor depending on the point, cf.~\eqref{def-weak-plane}. This in turn causes additional difficulties  and justifies the necessity of computations in Section 3,  as we are not allowed to straightforwardly mimic the existing results and methods.

To conclude this introduction let us comment on the higher dimensional counterparts of our studies. The strategy we adopt in this work is specific to the case of surfaces, with respect to both the complexification of the gradient used to prove the absence of critical points and the estimates for $L',L''$ computed in Section \ref{sec_isoper}. However, the main problem we consider, that is estimating the measure of level sets of ($a$-)harmonic functions, makes sense also on higher dimensional manifolds. In this direction, some related results with a different approach were obtained in \cite{fmp}, \cite{afm-arXiv}, \cite{bfm}.\\

\noindent{\bf Acknowledgements.} Part of the work was conducted when G.V. visited IMPAN and T.A. visited University of Milano-Bicocca. Both authors express their gratitude to the hosting institutions for their support, hospitality and creating the scientific atmosphere. Moreover, the authors would like to express the gratitude to Daniele Valtorta for suggesting the construction leading to assumption~\eqref{a-harm-ass3}. The second author is partially supported by Indam - GNAMPA.

\section{Isoperimetric inequality}\label{sec_isoper}
In this section we show counterparts of Alessandrini's isoperimetric inequality result~\cite[Theorem 1.1]{al2} for $a$-harmonic functions, as well as for the minimal surface equation, on two-dimensional Riemannian manifolds. However, we will formulate the main problem for all dimensions $n\geq 2$, as some of our results below can be applied in the general case of smooth Riemannian $n$-manifolds.
 
 Let $(M^n,g)$ be an $n$-dimensional Riemannian manifold  with Ricci curvature bounded from below: $\Ric\geq c$ for some fixed $c\in \R$. 
 
 The following auxiliary result is a counterpart of the well-known subharmonicity property for harmonic functions in the Euclidean setting. Furthermore, this observation is an important tool in the proof of Lemma~\ref{lem: L''}.

\begin{lem}\label{lem:log}
 Let $\Om$ be a domain in a $2$-dimensional Riemannian manifold with Gaussian curvature $K=K(x)$ and $u$ be a $C^3$-function in $\Om$. Then,  
 \begin{equation}\label{log-sub-ineq}
 \Delta (\log|\nabla u|)=\div\left(\frac{\Delta u}{|\nabla u|^2} \,\nabla u\right)+K.
 \end{equation}
at points, where $|\nabla u|\neq 0$. 
\end{lem}
In particular, if $u$ is harmonic, then we retrieve~\cite[Lemma 2.1]{adve}. 

Similar to the case of harmonic functions on surfaces, Lemma \ref{lem:log} follows from two well-known tools in geometric analysis. The first one is the Bochner formula for $C^2$ functions on $n$-dimensional Riemannian manifolds:
\begin{equation}\label{formula-Bochner}
	\Delta \frac{|\nabla u|^2}{2}=\langle \nabla \Delta u, \nabla u\rangle+|\nabla^2u|^2+\Ric(\nabla u, \nabla u).
\end{equation}
Secondly, we use the technique similar to the one in the proof of the refined Kato (in)equality, a standard formula in geometric analysis (see for instance \cite[p. 520]{lw} or \cite[Proposition 1.3]{PRS}  and references therein).  In dimension $2$, the refined Kato's inequality turns out to be an equality; see \eqref{refined-Kato2} and also~\cite[Lemma 2.2]{adve} for a complete proof. 
\begin{lem}\label{lem: kato}
	Let $\Om$ be a domain in an $n$-dimensional Riemannian manifold and $u$ be harmonic in $\Om$. Then, at points where $|\nabla u|\neq 0$, it holds that
	\begin{equation}\label{refined-Kato}
	|\nabla^2u|^2\geq \frac{n}{n-1}\left|\nabla|\nabla u|\right|^2,\quad n\geq 2.
	\end{equation}
	Moreover, if $n=2$ then one has indeed  
	\begin{equation}\label{refined-Kato2}
|\nabla^2u|^2=2\left|\nabla|\nabla u|\right|^2.
\end{equation}
\end{lem}

\begin{proof}[Proof of Lemma \ref{lem:log}] We generalize the proof of~\cite[Lemma 2.2]{adve} and use the same notation as therein. First, we directly compute that
\begin{align*}
\div\left(\frac{\Delta u}{|\nabla u|^2}\nabla u\right)&=\frac{\langle \nabla \Delta u, \nabla u\rangle}{|\nabla u|^2}+\frac{(\Delta u)^2}{|\nabla u|^2}-2\frac{{\rm Hess}(\nabla u, \nabla u)\Delta u}{|\nabla u|^4}.
\end{align*}
On the other hand, from the Bochner formula~\eqref{formula-Bochner} we obtain
\begin{align*}
\Delta(\log(|\nabla u|))=\div\left(\frac{\nabla |\nabla u|}{|\nabla u|}\right)&=\frac{\Delta |\nabla u|}{|\nabla u|}-\frac{|\nabla |\nabla u||^2}{|\nabla u|^2}\\
&=\frac{|\nabla^2u|^2-2\left|\nabla|\nabla u|\right|^2}{|\nabla u|^2}+K+\frac{\langle \nabla \Delta u, \nabla u\rangle}{|\nabla u|^2}.
\end{align*}
Therefore,
\begin{equation}\label{eq-ident-Kato}
 \div\left(\frac{\Delta u}{|\nabla u|^2}\nabla u\right)-\Delta(\log(|\nabla u|))=\frac{(\Delta u)^2}{|\nabla u|^2}-2\frac{{\rm Hess}(\nabla u, \nabla u)\Delta u}{|\nabla u|^4}-\frac{|\nabla^2u|^2-2\left|\nabla|\nabla u|\right|^2}{|\nabla u|^2}-K.
\end{equation}
In the case of harmonic functions, i.e. $\Delta u=0$, on the $2$-dimensional manifold, an application of the Kato equality \eqref{refined-Kato2} implies that all terms above containing $\Delta u$ vanish one by one, and so we retrieve the assertion of~\cite[Lemma 2.2]{adve}. In general this is not the case, except for dimension $2$, where the appropriate terms vanish when coupled together. Indeed, choose an orthonormal system diagonalizing the Hessian of $u$ at a fixed point, so that
${\rm Hess}\,u= {\rm diag}(\lambda_1,\lambda_2) {Id}_2$. Then we have 
\begin{align*}
|\nabla u|^2 (\Delta u)^2-2{\rm Hess}(u)(\nabla u, \nabla u)\Delta u&=(u_1^2+u_2^2)(\lambda_1+\lambda_2)^2-
2(\lambda_1u_1^2+\lambda_2u_2^2)(\lambda_1+\lambda_2)\\
&=  (u_1^2+u_2^2)(\lambda_1^2+\lambda_2^2)-2(\lambda_1^2u_1^2+\lambda_2^2u_2^2).
\end{align*}
Furthermore, since $\nabla |\nabla u|^2=2\nabla_{\nabla u}\nabla u$, it holds that
\begin{align*}
\left(|\nabla^2u|^2-2\left|\nabla|\nabla u|\right|^2\right)|\nabla u|^2&=|\nabla u|^2|\nabla^2 u|^2-\frac12 |\nabla |\nabla u|^2|^2\\
&=|\nabla u|^2|\nabla^2 u|^2-2|\nabla_{\nabla u}\nabla u|^2\\
&=(u_1^2+u_2^2)(\lambda_1^2+\lambda_2^2)-2(\lambda_1^2u_1^2+\lambda_2^2u_2^2).
\end{align*}
In a consequence, only the term $-K$ remains on the right hand side of \eqref{eq-ident-Kato}, and the assertion of the lemma follows.
\end{proof}
Suppose now that $\Om\Subset M^n$ is a topological annulus, i.e. a relatively compact domain homeomorphic to the Euclidean annulus $\{x\in\R^n\,:\,1<\|x\|<2\}\subset \R^n$, also called in the literature a ring domain, or a $2$-connected domain when $n=2$. In what follows, we will assume that the two connected boundary components $\Gamma_1$ and $\Gamma_2$ are $C^{1,\alpha}$, see the discussion before the proof of Theorem~\ref{thm-main1}. 
Recall that the measure of a level set of a function $v:\Om\to\R$ is given by
\begin{equation}\label{iso-in-length}
 L(t)=\int_{\{x\in \Om\,:\,v(x)=t\}}1\,{\rm d}\mathcal{H}^{n-1},
\end{equation}
where ${\rm d}\mathcal{H}^{n-1}$ stands for the {$(n-1)$-Hausdorff} measure.

In the next lemma we recall formulas allowing us to compute the first and the second derivatives of $L$ with respect to the height of the level curve of $C^3$-functions in a topological annulus $\Om\subset M^n$. The lemma generalizes Lemma 2.1 in~\cite{al2} for the planar $A$-harmonic case and Lemma 2.6 in ~\cite{adve} for the harmonic functions on smooth surfaces.
\begin{lem}\label{lem: L''}
Suppose that  $u:\Om\to\R$  is a $C^3$- function satisfying $|\nabla u|>0$ in a topological annulus $\Om\subset M^n$ and that $u$ attains constant boundary values, respectively, $u|_{\Gamma_1}=t_1$ and $u|_{\Gamma_2}=t_2$.
Then, the following holds for all $t_1<t<t_2$:
\begin{align}
 L'(t)&=\int_{\{x\in \Om\,:\,u(x)=t\}} \div\left(\frac{\nabla u}{|\nabla u|}\right)\frac{{\rm d}\mathcal{H}^{n-1}}{|\nabla u|}
 =\int_{\{x\in \Om\,:\,u(x)=t\}} \frac{\Delta u}{|\nabla u|^2}-\frac{\left \langle \nabla u, \nabla |\nabla u| \right \rangle}{|\nabla u|^3} {\rm d}\mathcal{H}^{n-1}. \label{L-der1} 
 \end{align}
Moreover, if $\Omega$ is an annulus in a $2$-dimensional manifold with Gauss curvature $K=K(x)$, then it holds that
 \begin{align}
 L''(t) &= \int_{\{x\in \Om\,:\,u(x)=t\}} \frac{1}{|\nabla u|}\left \langle \nabla \left(\frac{1}{|\nabla u|}\right), \frac{\Delta u}{|\nabla u|^2}\nabla u-\frac{\nabla|\nabla u|}{|\nabla u|} \right \rangle-\frac{K}{|\nabla u|^2} {\rm d}\mathcal{H}^1. \label{L-der2}   
 \end{align}
\end{lem}

\begin{rem}
 Notice that $C^3$-regularity assumption is not too much restrictive, since the $a$-harmonic functions (in particular $p$-harmonic ones) are smooth outside the set of critical points $\{|\nabla u|=0\}$; see e.g. the discussion on pg. 208 in~\cite{lew3} for the $p$-harmonic equation in $\R^n$, which in view of Proposition~\ref{lem-reg} below, can be extended to the setting of $a$-harmonic equations on surfaces.
\end{rem}

\begin{proof}
We follow the lines of the proof of~\cite[Lemma 2.6]{adve}, but for the readers convenience we recall its key steps, referring to~\cite{adve} for further details.
 
 Since $|\nabla u|>0$ by assumption, then $\nu=\frac{\nabla u}{|\nabla u|}$ is a unit vector normal to the level sets of $u$. Therefore, by the definition of the function $L$ in~\eqref{iso-in-length}, the Stokes theorem and the coarea formula, we have that
 \begin{align*}
  L'(t) &=\lim_{\ep \to 0}\frac{1}{\ep}\left(\int_{\{u=t+\ep\}} \left \langle \nu, \frac{\nabla u}{|\nabla u|} \right \rangle - \int_{\{u=t\}} \left \langle \nu, \frac{\nabla u}{|\nabla u|} \right \rangle \right) \\
  &= \lim_{\ep \to 0}\frac{1}{\ep} \int_{\{t<u<t+\ep\}} \div\left(\frac{\nabla u}{|\nabla u|}\right)=\int_{\{u=t\}} \frac{1}{|\nabla u|}\div\left(\frac{\nabla u}{|\nabla u|}\right) \\
&=\int_{\{u=t\}} \frac{1}{|\nabla u|^3}(\Delta u |\nabla u|-\langle \nabla u, \nabla|\nabla u|\rangle),
 \end{align*} 
 where the latter equality is obtained by direct computations, thus giving assertion~\eqref{L-der1}.
 
In order to show assertion~\eqref{L-der2} we first notice that
 \[
   \frac{\Delta u}{|\nabla u|^2}=\left \langle \frac{\Delta u \nabla u}{|\nabla u|^3}, \frac{\nabla u}{|\nabla u|} \right \rangle.
 \]
This observation together with computations involving~\eqref{L-der1}, the Stokes theorem and the coarea formula imply that:
 \begin{align}
 L''(t)&=\lim_{\ep \to 0}\frac{L'(t+\ep)- L'(t)}{\ep} \nonumber \\
 &= \lim_{\ep \to 0}\frac{1}{\ep} \int_{\{t<u<t+\ep\}} -\div\left(\frac{\nabla |\nabla u|}{|\nabla u|^2}-\frac{\Delta u \nabla u}{|\nabla u|^3}\right)\nonumber \\
 &=\lim_{\ep \to 0}\frac{1}{\ep} \int_{\{t<u<t+\ep\}} \div\left(\nabla \left(\frac{1}{|\nabla u|}\right)+\frac{\Delta u \nabla u}{|\nabla u|^3}\right)\nonumber \\
 &=\lim_{\ep \to 0}\frac{1}{\ep} \int_{\{t<u<t+\ep\}} \Delta\left(\frac{1}{|\nabla u|}\right)+\div\left(\frac{\Delta u \nabla u}{|\nabla u|^3}\right)\nonumber \\
 &=\int_{\{u=t\}} \left(\Delta\left(\frac{1}{|\nabla u|}\right)+\div\left(\frac{\Delta u \nabla u}{|\nabla u|^3}\right)\right)\frac{{\rm d}\mathcal{H}^{n-1}}{|\nabla u|}. \label{L-der2-aux}
   \end{align} 
 In order to handle the last integral, we employ Lemma~\ref{lem:log} and compute
 \begin{align}
 \div\left(\frac{\Delta u}{|\nabla u|^3} \nabla u\right)&=\frac{1}{|\nabla u|} \div\left(\frac{\Delta u}{|\nabla u|^2} \nabla u\right) +\left \langle \nabla \left(\frac{1}{|\nabla u|}\right),\nabla u \right \rangle \frac{\Delta u}{|\nabla u|^2} \nonumber \\
 &=\frac{1}{|\nabla u|}\Delta \log |\nabla u|-\frac{K}{|\nabla u|}+\left \langle \nabla \left(\frac{1}{|\nabla u|}\right),\nabla u \right \rangle \frac{\Delta u}{|\nabla u|^2} \nonumber \\
 &=\frac{1}{|\nabla u|}\left(\frac{\Delta |\nabla u|}{|\nabla u|}-\frac{|\nabla |\nabla u||^2}{|\nabla u|^2}\right)-\frac{K}{|\nabla u|}+\left \langle \nabla \left(\frac{1}{|\nabla u|}\right),\nabla u \right \rangle \frac{\Delta u}{|\nabla u|^2}.\label{L-der2-aux2}
 \end{align}
Since
\begin{equation}\label{L-der2-aux3}
 \Delta\left(\frac{1}{|\nabla u|}\right)=\div\left(-\frac{\nabla|\nabla u|}{|\nabla u|^2}\right)=-\frac{\Delta|\nabla u|}{|\nabla u|^2}+2\frac{|\nabla|\nabla u||^2}{|\nabla u|^3},
\end{equation}
then upon combining~\eqref{L-der2-aux2} and~\eqref{L-der2-aux3} we obtain the following identity
\begin{align}
 \div\left(\frac{\Delta u}{|\nabla u|^3} \nabla u\right)+ \Delta\left(\frac{1}{|\nabla u|}\right)&=\frac{|\nabla|\nabla u||^2}{|\nabla u|^3}-\frac{K}{|\nabla u|}+\left \langle \nabla \left(\frac{1}{|\nabla u|}\right),\nabla u \right \rangle \frac{\Delta u}{|\nabla u|^2} \nonumber \\
 &=-\left \langle \nabla \left(\frac{1}{|\nabla u|}\right), \frac{\nabla|\nabla u|}{|\nabla u|} \right \rangle-\frac{K}{|\nabla u|}+\left \langle \nabla \left(\frac{1}{|\nabla u|}\right),\frac{\Delta u}{|\nabla u|^2}\nabla u \right \rangle \nonumber \\
  &=\left \langle \nabla \left(\frac{1}{|\nabla u|}\right), \frac{\Delta u}{|\nabla u|^2}\nabla u-\frac{\nabla|\nabla u|}{|\nabla u|} \right \rangle-\frac{K}{|\nabla u|}.\label{L-der2-aux4}
\end{align}
We substitute~\eqref{L-der2-aux4} into ~\eqref{L-der2-aux} to obtain assertion~\eqref{L-der2}.
 \end{proof}

We are now in a position to present the proofs of Theorems~\ref{thm-main1} and~\ref{thm-main2}, the main results of this section and of the whole paper. 

Some comments about the statement of Theorem~\ref{thm-main1} are in order. Recall that  the rigidity obtained in case of equality in \eqref{e:logconv} or \eqref{e:logconv-p} holds only locally. We remark that the same understanding applies to the harmonic case, cf.~\cite[Theorem 2.7]{adve}. Namely, in the flat case $K\equiv 0$, the level curves and $\Om$ have the geometry of planar circles and circular annulus, respectively, only in the local sense. The locality in the assertion can not be avoided. Indeed, examples of topological annuli which verify the equality in \eqref{e:logconv} and \eqref{e:logconv-p} but are not proper global subsets of $\R^2$ can be obtained for instance as finite coverings of a regular circular annulus in the plane (where the solution of the Dirichlet problem is obtained by lifting the solution on the regular circular annulus).

In \cite[Theorem 2.10]{adve}, we showed that the nonpositivity of the curvature is also a necessary condition for the log-convexity of the length of the level curves of harmonic solutions. In the more general case of $a$-harmonic functions, this opposite direction has not been investigated yet. Namely, we conjecture that if the curvature is strictly positive at some point (hence, in some open set) of $M$, then one can find an annular domain for which the solution to ~\eqref{DP} does not satisfy \eqref{e:logconv-p}. In the attempt of mimicking the harmonic proof, the main additional difficulties is that $a$-harmonic functions are not conformally invariant, so that one would need very precise asymptotic estimates for the Green kernel of the operator. Even for the $p$-harmonic operator, the sharper estimates in the literature we are aware of (see e.g. \cite[2.4]{MSR}) are not precise enough to our purpose.
\smallskip

 The $C^{1,\alpha}$-regularity assumption on the boundaries of topological annuli in subject is a consequence of the interior ball condition assumed in the no-critical points lemma and, hence, required in the proof of Theorem~\ref{thm-main1}, cf.~the discussion in~\cite[Chapter 3.2]{gt}.  It is known that the $C^{1,\alpha}$-regularity characterizes domains with both interior- and exterior- ball conditions. Notice further that \cite{al2} assumes the $C^{2,\alpha}$-regularity. 

\begin{proof}[Proof of Theorem~\ref{thm-main1}]
 Let us notice that by the maximum and minimum principles for the $a$-harmonic function $u$, see~\cite[Theorem 8.5.1]{puse-book}, we have that $\max_{\overline{\Om}} u=t_2$ and  $\min_{\overline{\Om}} u=t_1$. Since $\nabla u\not=0$ due to Lemma \ref{non-zero-lemma}, the function $u$ is $C^\infty$ smooth and \eqref{def-weak-aharm} reads:
 \begin{equation}\label{lapl-mthm}
  \Delta u = -\frac{\langle \nabla u, \nabla (a(|\nabla u|_g)}{a(|\nabla u|_g)}=-\frac{a'(|\nabla u|_g)}{a(|\nabla u|_g)} \langle \nabla u, \nabla |\nabla u|_g \rangle.
 \end{equation}
Next, we observe that the Hopf lemma holds for solution of~\eqref{lapl-mthm} in $\Om$. Indeed, it follows from Theorem 2.8.3 in~\cite{puse-book} which requires the coefficient matrix of the operator in~\eqref{lapl-mthm} to be uniformly positive definite. This leads to analogous estimate for the ratio of eigenvalues $\Lambda_2/\Lambda_1$ as for $a_\ep$ in the proof of Lemma~\ref{lem-ALR} with the same lower and upper bounds as in~\eqref{est-eigen-ratio}. Thus, the Hopf lemma holds in our setting and we get that $|\nabla u|\geq const>0$ on $\partial \Om$. Since $\nabla u\not=0$ in $\Omega$ as noted above, by the smoothness of $u$ (in fact the $C^1$-regularity is enough), we have that there exists a positive constant $c$ such that $ \min_{\Om} |\nabla u|\ge c>0$.
  
By applying~\eqref{lapl-mthm} in formulas \eqref{L-der1}-\eqref{L-der2} for $L'$ and $L''$, we find the following equation and estimate, respectively:
\begin{align} 
L'(t) &=\int_{\{x\in \Om\,:\,u(x)=t\}} \frac{\Delta u}{|\nabla u|^2}-\frac{\left \langle \nabla u, \nabla |\nabla u| \right \rangle}{|\nabla u|^3}  
=-\int_{\{x\in \Om\,:\,u(x)=t\}} \frac{\left \langle \nabla u, \nabla |\nabla u| \right \rangle}{|\nabla u|^3}\left(1+\frac{a'(|\nabla u|)}{a(|\nabla u|)}|\nabla u|\right). \label{L-der1-est1} \\
 L''(t) &= \int_{\{x\in \Om\,:\,u(x)=t\}} \frac{1}{|\nabla u|}\left \langle -\frac{\nabla |\nabla u|}{|\nabla u|^2}, -\frac{a'(|\nabla u|)|}{a(|\nabla u|)} \left \langle \nabla |\nabla u|, \frac{\nabla u}{|\nabla u|} \right \rangle \frac{\nabla u}{|\nabla u|}-\frac{\nabla|\nabla u|}{|\nabla u|} \right \rangle-\frac{K}{|\nabla u|^2} \nonumber \\ 
 &= \int_{\{x\in \Om\,:\,u(x)=t\}} \frac{1}{|\nabla u|^4}\left \{ |\nabla |\nabla u||^2 +\frac{a'(|\nabla u|)|}{a(|\nabla u|)} \left \langle \nabla |\nabla u|, \frac{\nabla u}{|\nabla u|} \right \rangle^2 |\nabla u| \right \}-\frac{K}{|\nabla u|^2} \nonumber \\
 &\geq \int_{\{x\in \Om\,:\,u(x)=t\}} \frac{1}{|\nabla u|^4} \left \langle \nabla |\nabla u|, \frac{\nabla u}{|\nabla u|} \right \rangle^2 \left \{ 1 +\frac{a'(|\nabla u|)|}{a(|\nabla u|)}  |\nabla u| \right \}\qquad \scriptsize{(K\leq 0 \hbox{ and }|\nabla |\nabla u||^2\geq \langle \nabla |\nabla u|, \frac{\nabla u}{|\nabla u|} \rangle^2)} \nonumber \\
 &\geq \frac{1}{\beta}\int_{\{x\in \Om\,:\,u(x)=t\}} \frac{1}{|\nabla u|^4} \left \langle \nabla |\nabla u|, \frac{\nabla u}{|\nabla u|} \right \rangle^2 \left \{ 1 +\frac{a'(|\nabla u|)|}{a(|\nabla u|)}  |\nabla u| \right \}^2. \label{L-der2-est1}
\end{align}
Therefore, by the above estimates and the Cauchy--Schwarz inequality we have that
\begin{align*}
(L'(t))^2&=\left(\int_{\{x\in \Om\,:\,u(x)=t\}} \frac{\left \langle \nabla u, \nabla |\nabla u| \right \rangle}{|\nabla u|^3}\left(1+\frac{a'(|\nabla u|)}{a(|\nabla u|)}|\nabla u|\right)\right)^2 \\
&\leq L(t) \int_{\{x\in \Om\,:\,u(x)=t\}} \frac{\left \langle \nabla u, \nabla |\nabla u| \right \rangle^2}{|\nabla u|^6}\left(1+\frac{a'(|\nabla u|)}{a(|\nabla u|)}|\nabla u|\right)^2 \\
&\leq \beta L(t)L''(t).
\end{align*}
From this, assertions~\eqref{e:logconv} and~\eqref{e:logconv-p} of the theorem follow immediately.

	 In order to show the second part of the assertion, suppose that $(L^m(t))''=0$. This is equivalent to $\beta LL''=(L')^2$ which then by ~\eqref{L-der2-est1} (with the $K$-term remaining), implies
\begin{align*}
(L'(t))^2
&\leq L(t) \int_{\{x\in \Om\,:\,u(x)=t\}} \frac{\left \langle \nabla u, \nabla |\nabla u| \right \rangle^2}{|\nabla u|^6} \left(1+\frac{a'(|\nabla u|)}{a(|\nabla u|)}|\nabla u|\right)^2\\
&\leq \beta L(t) \int_{\{x\in \Om\,:\,u(x)=t\}} \left(\frac{1}{|\nabla u|^4}\left \{ |\nabla |\nabla u||^2 +\frac{a'(|\nabla u|)|}{a(|\nabla u|)} \left \langle \nabla |\nabla u|, \frac{\nabla u}{|\nabla u|} \right \rangle^2 |\nabla u| \right \}-\frac{K}{|\nabla u|^2}\right)\\
&=(L'(t))^2.
\end{align*}
Since $K\leq 0$, this chain of inequalities may hold only when $K\equiv 0$ in which case we reduce the discussion to the planar case and so Theorem 3.1 in~\cite{long} gives the second assertion of the theorem (see also \cite[Theorem 1.1]{al2} and~\cite[Theorems 2.7]{adve}). 
\end{proof}

It turns out that, at least in the $p$-harmonic case, the inequality in Theorem~\ref{thm-main1} can be quantified in the setting of surfaces with pinched curvature, provided that the solution defined on the annular domain can be extended to a positive $p$-harmonic function on a large enough ball containing the given annulus.

\begin{prop}\label{thm-main2-2}
	Let $(M^2,g)$ be a complete surface and suppose that its  Gauss curvature satisfies 
	\begin{equation}\label{assm-thm-main2-2}
		-\kappa _1\leq K \leq -\kappa_2 \leq 0\quad \text{in }B_{2R}
	\end{equation}
	for some $\kappa _1\geq \kappa_2\ge 0$ and some ball $B_{2R}\subset M^2$ or radius $2R>0$. Let $1<p<\infty$ and let $u>0$ be $p$-harmonic on $B_{2R}$. Suppose that $u$ takes constant values $0<t_1<t_2$ on the boundary components of a topological annulus with $C^{1,\alpha}$-boundary $\Omega\subset B_R\subset B_{2R}\subset M^2$. Then it holds that 
	\begin{align}
		&(\ln L(t))''\geq \frac{\kappa_2 }{\kappa _1}\frac{1}{t^2},\quad \hbox{for } t\in (t_1, t_2), \hbox{ if $p=2$},\label{est-pinched-1}\\
		&\left(\frac{p-1}{p-2} L^{\frac{p-2}{p-1}}(t)\right)''\geq \frac{R^2}{1+R}\frac{\kappa_2 }{1+R\kappa _1}\frac{1}{t^2}L^{-\frac{1}{p-1}}(t)\quad\hbox{for } t\in (t_1,t_2), \hbox{ if $p\not=2$}. \label{est-pinched-2}
	\end{align}
\end{prop}

\begin{proof} Let $u>0$ be $p$-harmonic in $\Om$ satisfying~\eqref{assm-thm-main2-2} and $\Om\subset B_R\Subset M^2$ be an annulus such that $u$ takes constant values, respectively $t_1$ and $t_2$, on the boundary components of $\Om$. By the formula~\eqref{L-der2-est1} and the inequality following it in the proof of Theorem~\ref{thm-main1}, we have that
\begin{align*}
	(L'(t))^2&\leq (p-1) L(t)\left(L''(t)+\int_{u=t}\frac{K}{|\nabla u|^2} d\mathcal H^1\right).
	\end{align*}
Hence, we obtain that if $p\not=2$, then
\begin{equation}\label{com-est-pinched}
 -L^{\frac{p}{p-1}}\left(\frac{p-1}{p-2} L^{\frac{p-2}{p-1}}(t)\right)''\leq \int_{u=t}\frac{K}{|\nabla u|^2} d\mathcal H^1.
\end{equation}
	
Then, Theorem 1.1 in~\cite{waz} applied on ball $B_{2R}\subset M^2$ and $\kappa:=\kappa_1$ asserts that
$|\nabla u|\leq C(p) \frac{1+\sqrt{\kappa _1}R}{R} |u|$ for all points in $B_R$. Therefore,
\[
 -\kappa_2\int_{u=t}\frac{1}{|\nabla u|^2} d\mathcal H^1\leq -\frac{1}{C(p)}\frac{R^2 \kappa_2 }{(1+R\sqrt{\kappa _1})^2}\int_{u=t}\frac{1}{|u|^2} d\mathcal H^1\leq-\frac{R^2 \kappa_2 }{(1+R)(1+R\kappa _1)}\frac{1}{t^2}L(t).
\]
 In the latter inequality we use an elementary estimate:
 \[
 (1+R\sqrt{\kappa_1})^2\leq R^2\kappa_1+1+R+R\kappa_1=(R+1)(1+R\kappa_1).
 \]
These combined with inequality~\eqref{com-est-pinched} give the assertion in the case $p\not=2$. The harmonic case follows from Proposition 2.13 in~\cite{adve}
\end{proof}

We conclude this section with the proof of Theorem~\ref{thm-main2}, which improves Theorem \ref{thm-main1} in the special case of minimal surface type equations, i.e., with the assumption \eqref{a-harm-ass-min} replacing~\eqref{a-harm-ass1}.

\begin{proof}[Proof of Theorem~\ref{thm-main2}]
First, let us observe that $\nabla u\not=0$ as Lemma \ref{non-zero-lemma} applies even if the assumption \eqref{a-harm-ass1} is replaced by \eqref{a-harm-ass-min}, so that in particular $\alpha$ may vanish. Indeed, Proposition \ref{lem-reg} is trivially verified in this setting, as $u$ is assumed to be $C^2$, while the discussion in Section \ref{Section3.1} holds also if $\alpha=0$.
	
Recall the formula for \emph{the curvature of the level curve} $k:=-\div\left(\frac{\nabla u}{|\nabla u|}\right)$ with respect to the outward normal vector (the sign depends on the assumption $t_1<t_2$). Moreover, observe that as in~\eqref{lapl-mthm} we may compute the laplacian of $u$. In a consequence we get 
\[
  \int_{\{x\in \Om\,:\,u(x)=t\}} k=\int_{\{x\in \Om\,:\,u(x)=t\}}-\div\left(\frac{\nabla u}{|\nabla u|}\right)=
  \int_{\{x\in \Om\,:\,u(x)=t\}} \frac{\left \langle \nabla u, \nabla |\nabla u| \right \rangle}{|\nabla u|^3}\left(1+\frac{a'(|\nabla u|)}{a(|\nabla u|)}|\nabla u|\right)|\nabla u|.
\]
Then, by~\eqref{L-der1-est1} and by the Cauchy--Schwarz inequality we obtain that
\begin{equation}\label{est-mins2}
 (L'(t))^2+\left(  \int_{\{x\in \Om\,:\,u(x)=t\}} k\right)^2\leq L(t)\left( \int_{\{x\in \Om\,:\,u(x)=t\}} \frac{\left \langle \nabla u, \nabla |\nabla u| \right \rangle^2}{|\nabla u|^6}\left(1+\frac{a'(|\nabla u|)}{a(|\nabla u|)}|\nabla u|\right)^2(1+|\nabla u|^2)\right).
\end{equation}
 Hence, by the growth assumption~\eqref{a-harm-ass-min} and estimates~\eqref{est-mins2} and~\eqref{L-der2-est1} we arrive at the following inequality
\begin{align}
 (L'(t))^2+\left(\int_{\{x\in \Om\,:\,u(x)=t\}} k\right)^2 &\leq L(t)\left( \int_{\{x\in \Om\,:\,u(x)=t\}} \frac{\left \langle \nabla u, \nabla |\nabla u| \right \rangle^2}{|\nabla u|^6}\left(1+\frac{a'(|\nabla u|)}{a(|\nabla u|)}|\nabla u|\right)\right) \nonumber\\
 & \leq L(t) L''(t). \label{est-mins3}
\end{align} 
To complete the argument note that since $\Om$ is an annular domain whose components $\Om_1$ and $\Om_2$ are homeomorphic to balls and $K\leq 0$, the Gauss--Bonnet theorem can be applied as follows
\[
 \int_{\{x\in \Om\,:\,u(x)=t\}} k=2\pi-\int_{{\rm Int }\{x\in \Om\,: u(x)=t\}} k\geq 2\pi.
\]
Here, we abuse notation and by ${\rm Int }\{x\in \Om\,: u(x)=t\}$ we denote the interior of the subset of $\Omega_1$ bounded by the level curve $\{x\in \Om\,: u(x)=t\}$. Then, upon applying this inequality at~\eqref{est-mins3} we arrive at assertion~\eqref{assert-thm-main2}. By the discussion analogous to the one in the end of the proof of Theorem~\ref{thm-main1}, we obtain that inequality in~\eqref{assert-thm-main2} holds only if $K\equiv 0$ and if the function $a(s)=1/\sqrt{1+s^2}$ for which in~\eqref{a-harm-ass-min} the equality holds. Moreover, equality has to hold in the Cauchy-Schwarz inequalities in \eqref{est-mins2}, so that $|\nabla u|$ and $\langle\nabla u,\nabla|\nabla u|\rangle$ have to be constant on $\{u=t\}$ for every $t\in(t_1,t_2)$. In particular, $k$ has to be constant and the level sets are concentric circles. Accordingly, $\Omega$ is locally isometric to a standard concentric annulus and by an explicit computation it turns out that $u$ must be locally the graph of a slice of catenoid.
\end{proof}
We remark that in the special case of the minimal surface equation, defined by $a(s)=1/\sqrt{1+s^2}$, in order to deduce that the solution has no critical points we do not need the whole machinery we introduced in Section \ref{sec_critical}. Indeed, the solution is a harmonic function with respect to the graph metric $du\otimes du$. Since the critical points of a function do not depend on the underlying metric, the easier techniques for harmonic functions apply, see \cite{adve}.

A similar technique as in the proof of Theorem~\ref{thm-main2} could also be applied to strictly spacelike solutions to the equation for maximal surfaces in the Lorentzian space, see Example \ref{ex-max-Lor}. In that case, one has that $1+\frac{a'(s)}{a(s)}s=(1-s^2)^{-1}$, and 
\begin{align*}
	\left(L'(t)+  \int_{\{x\in \Om\,:\,u(x)=t\}} k\right)^2
	&=	  \left(\int_{\{x\in \Om\,:\,u(x)=t\}} \frac{\left \langle \nabla u, \nabla |\nabla u| \right \rangle}{|\nabla u|^3}\frac{|\nabla u|-1}{1-|\nabla u|^2}\right)^2.
\end{align*}
As $|\nabla u|<1$, it holds $(1-|\nabla u|)^2 \le 1-|\nabla u|^2$, so that
\begin{align*}
	\left(L'(t)+  \int_{\{x\in \Om\,:\,u(x)=t\}} k\right)^2
	&\le	  L(t)\left(\int_{\{x\in \Om\,:\,u(x)=t\}} \frac{\left \langle \nabla u, \nabla |\nabla u| \right \rangle^2}{|\nabla u|^6}\frac{1}{1-|\nabla u|^2}\right)\le L(t)L''(t).
\end{align*}
In particular, if $M^2=\R^2$ then $\int_{\{x\in \Om\,:\,u(x)=t\}} k
=2\pi$ so that we have the relation
\[
\left( L'(t)+2\pi\right)^2 \le L(t)L''(t).
\]

\section{Critical points of $a$-harmonic functions on smooth surfaces}\label{sec_critical}

In this section we show that $a$-harmonic functions have isolated critical points on Riemannian surfaces, a property similar to the corresponding one for the $a$-harmonic (in particular $p$-harmonic) functions in the plane. The proof relies on the complex representation of the $a$-harmonic equation and on the associated regularity lemma which allows to reformulate the equation as a complex first order system of PDEs. Unlike the flat case of $\R^2$ the complex gradient need not be a quasiregular mapping, see pg. 6 in~\cite{strzelecki}. Nevertheless, the theory of complex first order systems permits us to conclude that the zeros of the gradient are isolated and form a discrete set of points, see~\cite{bjs}.

For the readers convenience we now recall some information stated in the preceding sections.

Let $(M^2, g)$ be a Riemannian surface and $\Om\subset M^2$ be an open set. Moreover, we consider function $a\in C^1(0,\infty)$ such that it satisfies the following assumptions:
\begin{align*}
&\hbox{\eqref{a-harm-ass1}}\quad  0<\alpha \leq 1+\frac{a'(s)s}{a(s)}\leq \beta,\quad \hbox{for all } s>0;\qquad \hbox{\eqref{a-harm-ass2}}\quad  sa(s)\to 0 \quad \hbox{ for }s\to 0, \\
&\hbox{\eqref{a-harm-ass3}}\quad \log a(s) \hbox{is either upper or lower bounded (or both) on $(0,1]$. }
\end{align*}

Recall that in the setting of Riemannian surfaces we may locally introduce the isothermal coordinates, denoted $z=(x,y)$, in which the metric $g$ takes the diagonal form with the conformal factor $\lambda>0$ a smooth, bounded and strictly positive function. Namely, $g(X,Y)=\lambda^2(x) \langle X, Y\rangle$ for any pair of vectors $X,Y$ at $x\in M^2$.
 
 We denote local bounds of $\lambda$ as follows: $0<c_g\leq \lambda<C_g<\infty$. Therefore, in the isothermal coordinates, we have that
\begin{equation}\label{grad-coord}
|\nabla u|_g=\lambda^{-1}|\nabla u|_0
\end{equation}
and equation~\eqref{def-strong-aharm}
in coordinates reads:
\begin{equation}\label{def-weak-plane}
\frac{\partial}{\partial x}\left(a(\lambda^{-1}(z)|\nabla u(z)|_0) u_x \right)+\frac{ \partial}{\partial y}\left(a(\lambda^{-1}(z)|\nabla u(z)|_0) u_y\right)=0,
\end{equation}
interpreted in the distributional sense.

Our next goal is to find the complex representation of ~\eqref{def-weak-plane}, following~\cite{strzelecki, alr}, and for this we need the auxiliary regularity observation, well known for $p$-harmonic functions in the plane and on smooth surfaces, as well as for planar $A$-harmonic equations with the $\delta$-monotonicity condition, see~\cite[Chapter 16]{aim} and also~\cite[Section 3]{dus}. 
\begin{prop}\label{lem-reg}
 Let $u$ be an $a$-harmonic function, i.e. satisfying~\eqref{def-weak-plane} in an open set $\Om\subset M$, under assumptions~\eqref{a-harm-ass1},  \eqref{a-harm-ass2} and~\eqref{a-harm-ass3}. Then it holds that $u\in W^{2,2}_{loc}(\Omega)$ and that
 \begin{equation}\label{a-delta-reg}
 a^{\delta}(|\nabla u|_{g})\nabla u\in W^{1,2}_{loc}(\Om)
 \end{equation}
 for any $\delta\in[0,1]$.
\end{prop}

In particular, when $\delta=1/2$, for the $p$-harmonic equation in $M$ we retrieve the assertion in~\cite{strzelecki}, namely that $|\nabla u|_g^{\frac{p-2}{2}}\nabla u\in W^{1,2}_{loc}(M)$ for $p\geq 2$. However, as explained later, our method allows to handle the $p$-harmonic functions for the whole range of $1<p<\infty$.

\begin{ex}\label{ex_valtorta}
	Even if the assumption \eqref{a-harm-ass3} is verified in all the significant examples, it is not automatically implied by \eqref{a-harm-ass1} and \eqref{a-harm-ass2}, so that we need to require it. Indeed, the following example, suggested to us by D. Valtorta, shows that one can find a positive function $a$ on $(0,\infty)$ satisfying \eqref{a-harm-ass1} and \eqref{a-harm-ass2}, but which is neither upper bounded nor bounded away from $0$. To this end, one can implement the changes of variables $t=-\log s$ and $f(t)=\log a(s)$. Then, \eqref{a-harm-ass1} and \eqref{a-harm-ass2} become
		\[
		1-\beta < \partial_t f(t) < 1-\alpha <1,\quad\text{and}\quad \lim_{t\to\infty} f(t)-t=-\infty.
		\] 
		A (two-side) unbounded function $f$ with this properties can be quite easily constructed, for instance, by smoothing out a piece-wise linear function $\tilde f$ which oscillates between $-t$ and $\sqrt{t}$, with the slope $\partial_t\tilde f(t)$ equal to either  $\frac{1}{2}$ or $-2$ outside corners.
\end{ex}

The following result is similar to Proposition 2.1 in~\cite{alr}, proved therein in the Euclidean setting, where $\lambda$ is constant.

\begin{lem}\label{lem-ALR}
 Let $\Om\subset M$ be a bounded connected open set and assume that $\partial \Om$ satisfies the interior and exterior ball condition. Let further $\phi\in C^2(\overline{\Om})$. If \eqref{a-harm-ass1} holds, then there exists a unique $u\in C^{1,\gamma}_{loc}\cap C(\overline{\Om})$ solving weakly the following Dirichlet problem:
\begin{equation*}
 \begin{cases}
 \div(a(|\nabla u|_g)\nabla u)=0 & \hbox{in } \Om,\\
 u|_{\partial \Om}=\phi|_{\partial \Om}.
 \end{cases}
\end{equation*}
 \end{lem}
 
\begin{proof}
The proof follows strictly the one in~\cite[Proposition 2.1]{alr} and therefore we will restrict our discussion only to the key differences in the surface setting.

Reasoning as in \cite[Lemma 3.6]{adve}, we can introduce a global isothermal coordinates chart $\phi:\Omega\to\R^2$. Accordingly, the problem ~\eqref{DP} can be reduced to the equation~\eqref{def-weak-plane} subject to a $C^2$ data, again denoted by $\phi$, in the plane.

Suppose that we know that a solution exists and is at least $C^1_{loc}(\Om)$ as proven in the further part of the discussion. Then the uniqueness follows from the comparison principle, see e.g. Theorem 2.4.1 and Proposition 2.4.3 in~\cite{puse-book} once we have checked that our $a$-harmonic equation satisfies the assumptions of that proposition. To this end, let $A(x, \xi):\Om\times \R^n\to \R^n$ be defined as follows $A(x, \xi):=a(\lambda^{-1}(x)|\mathbf \xi|_0) \mathbf \xi$. Then $A$ is continuous, since $\lambda>0$ and $a\in C^1(0,\infty)$. Next, we find the Jacobi matrix of $A$ with respect to $\xi$-variable:
\[
 D_{\xi}A(x,\xi)=\left[\delta_{ij}a(\lambda^{-1}(x)|\mathbf \xi|_0)+a'(\lambda^{-1}(x)|\mathbf \xi|_0)\frac{1}{\lambda}\frac{\xi_i\xi_j}{|\xi|_0}\right]_{ij},\qquad i,j=1,2.
\]
This together with the smoothness of $\lambda$ imply that $A\in C^1(\Om\times(\R^n\setminus\{0\}))$. Finally, we compute that 
\[
 \det D_{\xi}A(x,\xi)=a^2(\lambda^{-1}(x)|\mathbf \xi|_0)\left(1+\frac{a'(\lambda^{-1}(x)|\mathbf \xi|_0)\lambda^{-1}(x)|\mathbf \xi|_0}{a(\lambda^{-1}(x)|\mathbf \xi|_0)}\right)\geq \alpha a^2(\lambda^{-1}(x)|\mathbf \xi|_0)>0.
\]
Moreover, it holds that 
\[
 (D_{\xi}A(x,\xi))_{11}=a(\lambda^{-1}|\xi|_{0})\left[\left(\frac{a'(\lambda^{-1}|\xi|_{0})\lambda^{-1}|\xi|_{0}}{a(\lambda^{-1}|\xi|_{0})}+1\right)\frac{\xi_1^2}{|\xi|_0^2}+\frac{\xi_2^2}{|\xi|_0^2}\right]\!\!>a(\lambda^{-1}|\xi|_{0})\left[\frac{\alpha\xi_1^2+\xi_2^2}{|\xi|_0^2}\right]\!\!>0,\,\, \xi\not=0,
\]
and the similar estimate implies that $ (D_{\xi}A(x,\xi))_{22}>0$. All together, we conclude that matrices $D_{\xi}A(x,\xi)$ are positive definite on $\Om\times(\R^n\setminus\{0\})$. Therefore, Proposition 2.4.3 and Theorem 2.4.1 in~\cite{puse-book} can be applied to our $a$-harmonic operator.

 The existence and asserted regularity are proven by approximation of $a$ by regular functions $a_\epsilon$ for $\ep>0$ which satisfy a condition similar to~\eqref{a-harm-ass1} with slightly different bounds independent of $\epsilon$, i.e,
	\begin{equation}\tag{A$_\epsilon$}
		\qquad  0<\min\{\alpha;1\} \leq 1+\frac{a_{\ep}'(s)s}{a_{\ep}(s)}\leq \max\{\beta;1\},\quad \hbox{for all } s>0, \label{a-harm-ass1eps}
	\end{equation}
	 and such that $\inf a_\ep > c(\ep)>0$;  see pg. 197 in~\cite{alr} for details, in particular (2.2) and (2.3) therein. In a consequence we obtain a family of elliptic nondegenerate operators of corresponding Dirichlet problems
	  \begin{equation}\tag{DP$_\epsilon$}\label{DPeps}
	 	\begin{cases}
	 		\div(a_\epsilon(|\nabla u_\ep|_g)\nabla u_\ep)=0 & \hbox{in } \Om,\\
	 		u_\ep|_{\partial \Omega}=\phi,
	 	\end{cases}
	 \end{equation}
	and of associated solutions $u_\ep \in C^{2,\gamma}(\Omega)\cap C^0(\bar\Omega)$, see \cite[Theorem 12.5]{gt}. Instead of the family of equations (2.4a)' in~\cite{alr} we have  
\begin{equation}
 \Delta u_{\epsilon}+\lambda^{-2}\frac{a_\epsilon'(\lambda^{-1}|\nabla u_\epsilon|_0)}{a_\epsilon(\lambda^{-1}|\nabla u_\epsilon|_0)\lambda^{-1}|\nabla u_\epsilon|_0} \nabla u_\epsilon \nabla^2 u_\epsilon (\nabla u_\epsilon)^T -\frac{a_\epsilon'(\lambda^{-1}|\nabla u_\epsilon|_0)\,\lambda^{-1}\,|\nabla u_\epsilon|_0}{a_\epsilon(\lambda^{-1}|\nabla u_\epsilon|_0)}
 \left\langle \frac{\nabla \lambda}{\lambda}, \nabla u_\epsilon \right\rangle_0
 =0.\label{eq-eps}
\end{equation}
In our setting the resulting $a_\epsilon$ depend additionally on $z\in \Om$ through the presence of conformal factor $\lambda^{-1}$, cf.~\eqref{def-weak-plane}. However, since $\lambda>0$ is assumed to be bounded and smooth the discussion in~\cite{alr} stands true in our case as well. In order to apply~\cite[Theorem 12.5]{gt} we verify by direct computations that coefficients in ~\eqref{eq-eps} are defined and H\"older continuous. Moreover, the ratios of the eigenvalues of the coefficients matrix are uniformly bounded giving the uniform ellipticity. Indeed, upon denoting these eigenvalues by $\Lambda_1\leq\Lambda_2$, we find by~\eqref{a-harm-ass1eps} that
\begin{equation}\label{est-eigen-ratio}
1\leq \frac{\Lambda_2}{\Lambda_1}=\frac{2+\frac{a_\epsilon'(\lambda^{-1}|\nabla u_\epsilon|_0)\lambda^{-1}|\nabla u_\epsilon|_0}{a_\epsilon(\lambda^{-1}|\nabla u_\epsilon|_0)}+\left|\frac{a_\epsilon'(\lambda^{-1}|\nabla u_\epsilon|_0)\lambda^{-1}|\nabla u_\epsilon|_0}{a_\epsilon(\lambda^{-1}|\nabla u_\epsilon|_0)}\right|}{2+\frac{a_\epsilon'(\lambda^{-1}|\nabla u_\epsilon|_0)\lambda^{-1}|\nabla u_\epsilon|_0}{a_\epsilon(\lambda^{-1}|\nabla u_\epsilon|_0)}-\left|\frac{a_\epsilon'(\lambda^{-1}|\nabla u_\epsilon|_0)\lambda^{-1}|\nabla u_\epsilon|_0}{a_\epsilon(\lambda^{-1}|\nabla u_\epsilon|_0)}\right|}\leq \max\{\beta, \alpha^{-1} \}.
\end{equation}
Finally, the growth condition (iii) in~\cite[Theorem 12.5]{gt} on the first order term expression for us reads
\begin{equation}\label{est-gt125-iii}
\frac{|f|}{\Lambda_1}\leq \frac{\left|\frac{a_\epsilon'(\lambda^{-1}|\nabla u_\epsilon|_0)\lambda^{-1}|\nabla u_\epsilon|_0}{a_\epsilon(\lambda^{-1}|\nabla u_\epsilon|_0)}\right| |\nabla \lambda| \lambda^{-1}|\nabla u_\epsilon|_0}{\Lambda_1}\leq c(|\lambda|_{C^1(\Om)}, \min\{\alpha;1\},\max\{\beta;1\}) |\nabla u_\epsilon|_0.
\end{equation}
Therefore, \cite[Theorem 12.5]{gt} gives us the existence and desired regularity of $u_\ep$.
As in the proof of~\cite[Proposition 2.1]{alr} we may now apply \cite[Theorems 14.15 and 14.1]{gt} to get the uniform $C^0(\overline{\Om})$ and $C^{1,\gamma}_{loc}$  estimates for $u_\ep$. In particular \cite[Theorems 14.15]{gt} gives us the equicontinuity of $\{u_\ep\}_{\ep>0}$. Hence, the Ascoli--Arzel\`a theorem can be applied and we may conclude the assertion of the lemma. 

Applying~\cite[Theorems 14.15 and 14.1]{gt} reduces to checking that structure condition (14.9) in~\cite[Chapter 14.1]{gt} holds in our case. Namely, we need to verify that 
\begin{equation}\label{est-149}
|p|\Lambda_2+|f(x)|\leq \mu\mathcal{E}(x,p) \quad \hbox{for all }(x,p)\in\Om\times\R^n,
\end{equation}
 with $|p|\geq \mu$. Here $\mathcal{E}(x,p):=\sum_{i,j=1,2}{a_{\ep}}_{ij}(x,p)p_ip_j$ is the quadratic form defined by the coefficients matrix of $a_{\ep}$. By direct computations we find that
\begin{align*}
\mathcal{E}(x,p)&=\left(1+\lambda^{-1}\frac{a'_{\ep}(\lambda^{-1}|p|)}{a_{\ep}(\lambda^{-1}|p|)}\frac{p_1^2}{|p|}\right)p_1^2+2\lambda^{-1}\left(\frac{a'_{\ep}(\lambda^{-1}|p|)}{a_{\ep}(\lambda^{-1}|p|)}\frac{p_1p_2}{|p|}\right)p_1p_2+\lambda^{-1}\left(1+\frac{a'_{\ep}(\lambda^{-1}|p|)}{a_{\ep}(\lambda^{-1}|p|)}\frac{p_2^2}{|p|}\right)p_2^2\\
&=|p|^2+\lambda^{-1}\frac{a'_{\ep}(\lambda^{-1}|p|)}{a_{\ep}(\lambda^{-1}|p|)}|p|^3
\geq \min\{\alpha,1\}|p|^2.
\end{align*}
On the other hand, by the estimate of $|f|$ in~\eqref{est-gt125-iii} we have that $|p|\Lambda_2+|f|\leq |p|(\Lambda_2+c(|\lambda|_{C^1(\Om)}, \alpha,\beta))$ (here $c$ denotes a possibly different constant which also includes $\Lambda_1$). Therefore, by setting $\mu:=\sqrt{\frac{\Lambda_2+c(|\lambda|_{C^1(\Om)}, \alpha,\beta)}{\min\{\alpha,1\}}}$, we get that for $|p|\geq \mu$ condition~\eqref{est-149} holds true, justifying the use of~\cite[Theorems 14.15]{gt}. Hence the proof is completed.
\end{proof}

\begin{proof}[Proof of Proposition~\ref{lem-reg}]
	
First, suppose that $a$ is upper bounded. By the discussion in the proof of Lemma~\ref{lem-ALR} we have that the solution of the Dirichlet problem~\eqref{DP} for \eqref{def-weak-plane}, and hence for \eqref{def-weak-aharm}, is continuous in $\overline{\Om}$. Let now $U\Subset V \Subset \Om$ be smooth domains and consider the same $\epsilon$-regularization $a_\ep$ of $a$ as in the proof of the previous lemma. For such a family of operators we solve the Dirichlet problems~\eqref{DPeps} in $V$ subject to continuous boundary data $u|_{\partial V}$. Then, by the discussion on pg. 198 in~\cite{alr}
the sequence $(u_\epsilon)$ is uniformly bounded in $C_ {loc}^{1,\alpha}(V)$ and converges, up to a subsequence, in $C^1(V)\cap C^0(\bar V)$ to the unique solution $u$ to the Dirichlet problem on $V$.  
To each one of the $u_\ep$ we apply Theorem \ref{th:regularity} to deduce that $\|u_\ep\|_{W^{2,2}(U)}$ is bounded, independently of $\ep$. Hence, a subsequence of $(u_\ep)$ converges in 
$W^{2,2}(U)$ to a limit function $u_0\in W^{2,2}(U)$, and necessarily $u_0=u$. In particular $u\in W^{2,2}_{loc}$.

In order to prove the second assertion of the proposition, we first compute 
	\begin{align*}
	& \partial_x\left(a^{\delta}_{\epsilon}(\lambda^{-1}|\nabla {u_\epsilon}|_{0})u_x\right)=a_\ep^{\delta}(\lambda^{-1}|\nabla {u_\epsilon}|_{0})\Bigg( \left[\delta \cA_\ep \frac{{u_\epsilon}_x^2}{|\nabla u_\epsilon|^2_{0}}+1\right]{u_{\epsilon}}_{xx}+\delta \cA_\ep \frac{{u_\epsilon}_x{u_\ep}_y}{|\nabla u_\epsilon|^2_{0}}{u_{\epsilon}}_{xy}
-\delta \frac{\lambda_x}{\lambda} \cA_\ep {u_\epsilon}_x\Bigg), 
\\&\partial_x\left(a^{\delta}_{\epsilon}(\lambda^{-1}|\nabla u_\epsilon|_{0})u_y\right)=a_\ep^{\delta}(\lambda^{-1}|\nabla u_\epsilon|_{0})\Bigg( \left[\delta \cA_\ep\frac{{u_\epsilon}_y^2}{|\nabla u_\epsilon|_{0}}+1\right]{u_{\epsilon}}_{xy}+\delta 	\cA_\ep\frac{{u_\epsilon}_x{u_\ep}_y}{|\nabla u_\epsilon|^2_{0}}{u_\epsilon}_{xx}
-\delta\frac{\lambda_x}{\lambda} \cA_\ep {u_\epsilon}_y\Bigg), 
\end{align*}
where $\cA_\ep=\frac{a_\ep'(\lambda^{-1}|\nabla u_\epsilon|_{0})\lambda^{-1}|\nabla u_\epsilon|_{0}}{a_\ep(\lambda^{-1}|\nabla u_\epsilon|_{0})}$. Similar expressions hold for $\partial_y$. Now, $|\cA_\ep|$ is bounded independently of $\ep$, and for every $\ep$ small enough there exists a constant $C_a$ such that $0<a_\ep\le C_a$ on $(0,u^\ast]$, with $u^\ast:=\sup_\ep\|\nabla u_\ep\|_{L^\infty(V)}$. Hence, the $a_{\epsilon}^{\delta}(\lambda^{-1}|\nabla u_\epsilon|_{0})\nabla u_\epsilon$ are uniformly bounded in $W^{1,2}(U, \R^2)$ and thus one of its subsequences converges in $W^{1,2}(U, \R^2)$ to a vector field $X\in W^{1,2}(U, \R^2)$. 

In order to conclude the first part of the proof of~\eqref{a-delta-reg}, we \textsl{claim} that $a_{\epsilon}^{\delta}(\lambda^{-1}|\nabla u_\epsilon|_{0})\nabla u_\epsilon$ converges to $a^{\delta}(\lambda^{-1}|\nabla u|_{0})\nabla u$ point-wisely, so that
$X=a^{\delta}(\lambda^{-1}|\nabla u|_{0})\nabla u$.  Indeed
\begin{align*}
	\left|a_{\epsilon}^{\delta}(\lambda^{-1}|\nabla u_\epsilon|_{0})\nabla u_\epsilon-a^{\delta}(\lambda^{-1}|\nabla u|_{0})\nabla u\right|
	&\le 
	\left|a_{\epsilon}^{\delta}(\lambda^{-1}|\nabla u_\epsilon|_{0})-a^{\delta}(\lambda^{-1}|\nabla u_\epsilon|_{0})\right|\, \left|\nabla u_\epsilon\right|\\
	&+\left|a^{\delta}(\lambda^{-1}|\nabla u_\epsilon|_{0})\nabla u_\epsilon-a^{\delta}(\lambda^{-1}|\nabla u|_{0})\nabla u\right|.
\end{align*}
The second term on the right-hand side converges to $0$ since $\nabla u_\ep \to \nabla u$ uniformly on $U$ and $a^{\delta}(s)s=(a(s)s)^{\delta}s^{1-\delta}$ is continuous on $[0,\infty)$ by assumption \eqref{a-harm-ass2}. 
The first term on the right-hand side converges to $0$ since $a_\ep^\delta\to a_\ep$ in $C^1_{loc}(0,+\infty)$ by the construction in~\cite{alr}, while
$[a_\ep^{\delta}(s)-a^{\delta}(s)]\,s\le 2C_a^\delta t$ on $(0,t]$. 

Suppose now that $a$ is not upper bounded in a neighborhood of $0$. Then $1/a$ is.
By mimicking the stream function method for $p$-harmonic equation in the plane we will show that in such a case assertion of the lemma holds as well, see \cite[Chapter 16.1, Theorem 16.3.1]{aim} and~\cite{arli}. 

Let $U\Subset \Om$ be simply-connected. We define a function $v:U\to \R$ as a solution of the equation
\[
 \nabla v=\star (a(|\nabla u|_g)\nabla u).
\]
Such a solution exists, as the vector field $\star (a(|\nabla u|_g)\nabla u)$ is irrotational in the simply connected set $U$, hence conservative. In local coordinates this corresponds to the following system of PDEs:
\[
 \begin{cases}
 v_x=-a(\lambda^{-1}(z)|\nabla u(z)|_0) u_y\\
 v_y=a(\lambda^{-1}(z)|\nabla u(z)|_0) u_x. 
 \end{cases}
\]
From this, we get $|\nabla v|_0=a(\lambda^{-1}(z)|\nabla u(z)|_0)|\nabla u|_0$ and hence
\begin{equation}\label{eq-stream}
 |\nabla v|_g=a(|\nabla u|_g)|\nabla u|_g.
\end{equation}
Define function $F:(0,\infty)\to \R_+$ as follows $F(t):=a(t)t$. Since,
\[
 F'(t)=a\left(1+\frac{a'(t)t}{a}\right),
\]
we have that, by assumption~\eqref{a-harm-ass1}, $0<\alpha a(t)\leq F'(t)\leq \beta a(t)$ for all $t>0$. Therefore, the inverse of $F$ exists and $|\nabla u|_g=F^{-1}(|\nabla v|_g)$. Moreover, observe that by~\eqref{eq-stream} we have that $|\nabla u|_g=0$ if and only if $|\nabla v|_g=0$, and so the sets of critical points for $u$ and $v$ are the same. We directly check that $v$ satisfies the following equation:
\begin{equation}\label{eq-conj}
 \div\left(\frac{1}{a(F^{-1}(|\nabla v|_g))}\nabla v\right)=0.
\end{equation}
Define function $b$ as follows $b(t):=\frac{1}{a(F^{-1}(t))}$. It holds that $b(t)>0$ for $t>0$ and, moreover, $b$ satisfies assumption~\eqref{a-harm-ass1}. Indeed, it holds that
\[
b'(t)=-\frac{a'(F^{-1}(t))}{a(F^{-1}(t))^2} \frac{d}{ds}\left(\frac{1}{F(s)}\right)\bigg|_{s=F^{-1}(t)}=-\frac{a'(F^{-1}(t))}{a(F^{-1}(t))^3}\,\frac{1}{1+\frac{a'(s)s}{a(s)}}\bigg|_{s=F^{-1}(t)}.
\]
Henceforth, 
\[
 1+\frac{b'(t)t}{b(t)}=1-\frac{a'(s)s}{a(s)}\frac{1}{1+\frac{a'(s)s}{a(s)}}=\frac{1}{1+\frac{a'(s)s}{a(s)}}.
\]
Therefore, $b$ satisfies~\eqref{a-harm-ass1} with $\alpha'=\alpha^{-1}$ and $\beta'=\beta^{-1}$. Moreover, since by assumption $1/a$ is bounded in a neighbourhood of $0$, function $a>0$ and satisfies~\eqref{a-harm-ass2}, we have that $tb(t)\to 0$, as $t\to 0$, giving that $b$ satisfies~\eqref{a-harm-ass2}. Furthermore, since $a$ is unbounded in the neighbourhood of $0$, then $b$ is bounded and Lemma~\ref{lem-ALR} can be applied to $b$. In a consequence we get that $b^{1-\delta}(|\nabla v|_{g})\nabla v\in W^{1,2}_{loc}(\Om)$ for $v$ solving~\eqref{eq-conj}. However, in the local coordinates we have that
\[
 b^{1-\delta}(|\nabla v|_{g})\nabla v=\frac{1}{a^{1-\delta}(F^{-1}(|\nabla v|_g))}\star (a(|\nabla u|_g)\nabla u)=\left(-a^{\delta}(\lambda^{-1}(z)|\nabla u(z)|_0) u_y, a^{\delta}(\lambda^{-1}(z)|\nabla u(z)|_0) u_x\right).
\]
This implies that also $a^{\delta}(|\nabla u|_{g})\nabla u\in W^{1,2}_{loc}(\Om)$ even in case $a$ is lower bounded away from $0$ in a neighborhood of $0$, but not necessarily upper bounded. 
\end{proof}

\begin{ex} Let $a(s)=s^{p-2}$ for $1<p<\infty$, then $F^{-1}(t)=t^{\frac{1}{p-1}}$ and $a(F^{-1}(|\nabla v|_g))=|\nabla v|_g^{\frac{p-2}{p-1}}$. Therefore, the conjugate equation of the $p$-harmonic one is $q$-harmonic for $q=\frac{p}{p-1}$, as
	\[
	\div(|\nabla v|_g^{\frac{2-p}{p-1}}\nabla v) =  \div(|\nabla v|_g^{\frac{p}{p-1}-2}\nabla v) = \div(|\nabla v|_g^{q-2}\nabla v)=0.
	\]
Similarly, let  $a(s)=(1+s^2)^{-1/2}$, which corresponds to the minimal surface equation. Then $F^{-1}(t)=\sqrt{\frac{t^2}{1-t^2}}$ and 
	\[
	\div({b(|\nabla v|_g)}\nabla v  )=\div\left(\frac{1}{a(F^{-1}(|\nabla v|_g))}\nabla v  \right)=\div\left(\frac{1}{\sqrt{1-|\nabla v|_g^2}}\right),
	\]
i.e., the maximal graph equation in Lorentzian spacetime.
\end{ex}

\subsection{Complex representation of $a$-harmonic equation on surfaces}\label{Section3.1}

Let us now pass to finding the complex system of equations corresponding to \eqref{def-weak-aharm}. In order to complete this goal we will follow the standard approach, see e.g.~\cite{strzelecki} for the setting of $p$-harmonic functions on surfaces and~\cite[Section 3]{alr} for the setting of planar $a$-harmonic functions.

Recall that the complex gradient of $u$ can be defined in local coordinates as $f:=u_x-iu_y$ and the associated operator is
\begin{equation}\label{e:F}
 F(z):=a^{\frac12}\Big(\lambda(z)^{-1}|f(z)|\Big)f(z).
\end{equation}

Since, in the distributional sense, it holds that $u_{xy}=u_{yx}$ and $F\in W^{1,2}_{loc}$ by Proposition \ref{lem-reg}, we have that 
\begin{equation}\label{eq-mixed-der}
 \frac{\partial}{\partial y}\left(\frac{F+\barF}{a^{\frac12}(\lambda^{-1}|f|)}\right)=i\, \frac{\partial}{\partial x}\left(\frac{F-\barF}{a^{\frac12}(\lambda^{-1}|f|)}\right),
\end{equation}
in the sense of distributions. Recall that $\frac{\partial}{\partial z}:=\frac12(\frac{\partial}{\partial x}-i\,\frac{\partial}{\partial y})$ and $\frac{\partial}{\partial \bz}:=\frac12(\frac{\partial}{\partial x}+i\,\frac{\partial}{\partial y})$. Using this notation we rewrite~\eqref{eq-mixed-der} so that the following holds in the distributional sense:
\[
 \frac{\partial}{\partial \bz}\left(\frac{F}{a^{\frac12}(\lambda^{-1}|f|)}\right)= \frac{\partial}{\partial z}\left(\frac{\barF}{a^{\frac12}(\lambda^{-1}|f|)}\right).
\]
Equivalently this reads 
\begin{equation}\label{eq-1cmplx-1}
 F_{\bz}-\overline{F_{\bz}}=\frac{(a^{1/2})_{\bz}}{a^{1/2}}F-\frac{(a^{1/2})_{z}}{a^{1/2}}\barF.
\end{equation}
Next, we express the above equation in terms of $F$ and related expressions. Note that
\begin{equation}\label{eq-aA}
 \lambda^{-1}|F|=a^{\frac12}(\lambda^{-1}|f(z)|)\lambda^{-1}|f(z)|.
\end{equation}
As above we find that the inverse function of $A(t)=a^{\frac12}(t)t$ exists due to $a$ satisfying assumption \eqref{a-harm-ass1}, and thus $\lambda^{-1}|f(z)|=A^{-1}(\lambda^{-1}|F|)$. This implies that
\[
 a(\lambda^{-1}|f(z)|)=a(A^{-1}(\lambda^{-1}|F|)).
\]
Therefore, we may rewrite~\eqref{eq-1cmplx-1} as follows:
\begin{align}
 F_{\bz}-\overline{F_{\bz}}&= \frac{(a^{\frac12}(A^{-1}(\lambda^{-1}|F|))_{\bz}}{a^{\frac12}}\,F - \frac{(a^{\frac12}(A^{-1}(\lambda^{-1}|F|))_{z}}{a^{\frac12}}\,\overline{F} \nonumber \\
 &=\frac{1}{2a} \left(a'(A^{-1})\,\left[ A^{-1}\right]'\right)\Big|_{\lambda^{-1}|F|} \left\{\left((\lambda(z)^{-1}|F(z)|)\right)_{\bz}\,F-  \left((\lambda(z)^{-1}|F(z)|)\right)_{z}\,\overline{F}\right\} \nonumber \\
&=\frac{1}{4a} \left(a'(A^{-1})\,\left[ A^{-1}\right]'\right)\Big|_{\lambda^{-1}|F|}\lambda^{-1}|F|
\left\{F_{\bz}-\overline{F_{\bz}}+ \frac{F}{\bar{F}}\overline{F_z}-\frac{{\bar F}}{F}F_z-2\frac{F}{\lambda}\lambda_{\bz}+2\frac{\bar{F}}{\lambda}\lambda_{z}\right\}. \label{eq-1cmplx}
\end{align}
Upon denoting by $B:=\frac{1}{4a} \left(a'(A^{-1})\,\left[ A^{-1}\right]'\right)\big|_{\lambda^{-1}|F|}\lambda^{-1}|F|$, we solve the equation for $F_{\bz}-\overline{F_{\bz}}$ to get
\begin{equation}\label{eq-11cmplx}
 F_{\bz}-\overline{F_{\bz}}=\frac{B}{1-B}\left\{\frac{F}{\bar{F}}\overline{F_z}-\frac{{\bar F}}{F}F_z-\frac{F}{\lambda}\lambda_{\bz}+\frac{\bar{F}}{\lambda}\lambda_{z}\right\}.
\end{equation}

\begin{ex}
 If $a(s)=s^{p-2}$, then $A(s)=a^{\frac12}(s)s=s^{\frac{p}{2}}$ and by the direct calculations we find that $\frac{B}{1-B}=\frac{p-2}{p+2}$. Therefore, we retrieve the $p$-harmonic case in~\cite[Formula (2.6)]{strzelecki}.
\end{ex}

On the other hand the $a$-harmonic equation can be written as follows:
\[
 \frac{\partial}{\partial x}\left((F+\barF) a^{\frac12}(\lambda^{-1}|f|)\right)+i\, \frac{\partial}{\partial y}\left((F-\barF) a^{\frac12}(\lambda^{-1}|f|)\right)=0,
\]
which, using the complex derivative, reads:
\[
  \frac{\partial}{\partial \bz}\left(F a^{\frac12}(\lambda^{-1}|f|)\right) + \frac{\partial}{\partial z}\left( \barF a^{\frac12}(\lambda^{-1}|f|)\right)=0.
\]
By~\eqref{eq-aA} and the discussion following it, we arrive at the equation
\begin{equation}\label{eq-2cmplx}
  \frac{\partial}{\partial \bz}\left(F \frac{\lambda^{-1}|F|}{A^{-1}(\lambda^{-1}|F|)}\right) + \frac{\partial}{\partial z}\left(\barF \frac{\lambda^{-1}|F|}{A^{-1}(\lambda^{-1}|F|)}\right)=0.
\end{equation}

Upon direct differentiation equation~\eqref{eq-2cmplx} becomes
\begin{align*}
 0&=(F_{\bz}+\overline{F_{\bz}})A^{-1}(\lambda^{-1}|F|)\lambda^{-1}|F|\\
 &\qquad+A^{-1}(\lambda^{-1}|F|)\left[\left((\lambda(z)^{-1}|F(z)|)\right)_{\bz}\,F+ \left((\lambda(z)^{-1}|F(z)|)\right)_{z}\,\overline{F}\right]\\
&\qquad -(A^{-1})'(\lambda^{-1}|F|)\lambda^{-1}|F|\left[\left((\lambda(z)^{-1}|F(z)|)\right)_{\bz}\,F+  \left((\lambda(z)^{-1}|F(z)|)\right)_{z}\,\overline{F}\right]\\
&=(F_{\bz}+\overline{F_{\bz}})A^{-1}(\lambda^{-1}|F|)\lambda^{-1}|F|\\
&\qquad +\frac12\left[A^{-1}(\lambda^{-1}|F|)-(A^{-1})'(\lambda^{-1}|F|)\lambda^{-1}|F|\right]\lambda^{-1}|F|\left\{F_{\bz}+\overline{F_{\bz}}+ \frac{F}{\bar{F}}\overline{F_z}+\frac{{\bar F}}{F}F_z-2\frac{F}{\lambda}\lambda_{\bz}-2\frac{\bar{F}}{\lambda}\lambda_{z}\right\}.
\end{align*}
Similarly to ~\eqref{eq-11cmplx}, we solve the last equation for $F_{\bz}+\overline{F_{\bz}}$ and arrive at the following one
\begin{equation}\label{eq-22cmplx}
F_{\bz}+\overline{F_{\bz}}=C\left\{\frac{F}{\bar{F}}\overline{F_z}+\frac{{\bar F}}{F}F_z-2\frac{F}{\lambda}\lambda_{\bz}-2\frac{\bar{F}}{\lambda}\lambda_{z}\right\},
\end{equation}
where
\[
 C:=\frac{(A^{-1})'(\lambda^{-\frac12}|F|)\lambda^{-\frac12}|F|-A^{-1}(\lambda^{-\frac12}|F|)}{3A^{-1}(\lambda^{-\frac12}|F|)-(A^{-1})'(\lambda^{-\frac12}|F|)\lambda^{-\frac12}|F|}.
\]
We add up ~\eqref{eq-11cmplx} and~\eqref{eq-22cmplx} to obtain the following  equation:
\begin{equation}\label{eq-final-cmplx}
 F_{\bz}-a_1F_z-a_2\overline{F_z}=-2a_1\overline{F}\frac{\lambda_{z}}{\lambda}-2a_2F\frac{\lambda_{\bz}}{\lambda},
\end{equation}
with $a_1:=\frac12(C-\frac{B}{1-B})\frac{\overline{F}}{F}$ and $a_2:=\frac12(C+\frac{B}{1-B})\frac{F}{\overline{F}}$.
It remains to prove that
\begin{equation}\label{eliptic-cond}
\|a_1\|_{L^{\infty}(\Om)}+\|a_2\|_{L^{\infty}(\Om)}<1,
\end{equation}
which implies the uniform ellipticity of~\eqref{eq-final-cmplx}. 

First, let $A(s)=a^{\frac12}(s)s$ and notice that $(A^{-1}(t))'=\frac{1}{A(s)'}$ at $s=A^{-1}(t)$, which gives that
\[
(A^{-1}(t))'=\frac{1}{a^{\frac12}(s)} \frac{1}{\frac12\frac{a'(s)s}{a(s)}+1}.
\]
Hence ($s=A^{-1}(t)$)
\begin{equation}\label{form-A-inv}
	\frac{(A^{-1}(t))'t}{A^{-1}(t)}=\frac{1}{a^{\frac12}(s)} \frac{1}{\frac12\frac{a'(s)s}{a(s)}+1}\frac{A(s)}{s}=\frac{1}{\frac12\frac{a'(s)s}{a(s)}+1}.
\end{equation}
Setting $D=D(s)=\frac{a'(s)s}{a(s)}$, we have 
$
\frac{(A^{-1}(t))'t}{A^{-1}(t)}=\frac{1}{\frac12D+1}$, from which
\begin{align*}
	B&=\frac14 \frac{a'(A^{-1})}{a(A^{-1})}A^{-1}\,\frac{1}{\frac12\frac{a'(s)s}{a(s)}+1} = \frac 14 D \frac{1}{\frac{1}{2}D+1}=\frac 12 \frac{D}{D+2}.
\end{align*}
Moreover, from \eqref{form-A-inv},
$$C=\frac{\frac{(A^{-1}(t))'t}{A^{-1}(t)}-1}{3-\frac{(A^{-1}(t))'t}{A^{-1}(t)}}=\frac{\frac{1}{\frac12D+1}-1}{3-\frac{1}{\frac12D+1}}=\frac{-D}{3D+4}.
$$
Hence,
\[
\frac{B}{1-B}=\frac{D}{D+4},\quad
C+\frac{B}{1-B}=\frac{2D^2}{(3D+4)(D+4)},\quad
C-\frac{B}{1-B}=-4D\frac{D+2}{(3D+4)(D+4)}.
\]
Note that $C+B/(1-B)>0$ as, 
by assumption \eqref{a-harm-ass1}, $-1<\alpha-1\le D \le \beta-1$, while $C-B/(1-B)>0$ if and only if $D<0$. In particular
\[
\left|C+\frac{B}{1-B}\right|+\left|C-\frac{B}{1-B}\right| = \frac{2D^2}{(3D+4)(D+4)} -4D\frac{D+2}{(3D+4)(D+4)}=\frac{-8D-2D^2}{(3D+4)(D+4)}<2  
\]
if $-1<D<0$, while
\[
\left|C+\frac{B}{1-B}\right|+\left|C-\frac{B}{1-B}\right| = \frac{2D^2}{(3D+4)(D+4)} +4D\frac{D+2}{(3D+4)(D+4)}=\frac{6D^2+8D}{(3D+4)(D+4)}<2
\]
if $D\ge 0$. Thus, under the  growth condition~\eqref{a-harm-ass1}, the inequality~\eqref{eliptic-cond} is proved and the uniform ellipticity of~\eqref{eq-final-cmplx} follows.

\begin{rem}
		Since $a^\delta(\lambda(z)^{-1}|f(z)|)f$ is in $W^{1,2}_{loc}$ for any $\delta\in[0,1]$, we can repeat the argument above for different values of the exponent. For instance, take $\delta=1$ and define $G(z):=a\left(\lambda(z)^{-1}|f(z)|\right)f(z)$. Then \eqref{eq-1cmplx-1} reads

\begin{equation}
 G_{\bz}-\overline{G_{\bz}}=\frac{a_{\bz}}{a}G-\frac{a_{z}}{a}\barG.
\end{equation}
We define $A(t):=a(t)t$ and repeat computations as in \eqref{eq-1cmplx}. The equation corresponding to~\eqref{eq-11cmplx} takes the following form:
\[
G_{\bz}-\overline{G_{\bz}}=\frac{B'}{1-B'}\left\{\frac{G}{\bar{G}}\overline{G_z}-\frac{{\bar G}}{G}G_z-\frac{G}{\lambda}\lambda_{\bz}+\frac{\bar{G}}{\lambda}\lambda_{z}\right\},
\]
where
\[
B':=\frac{1}{2a(A^{-1}(w))} \left(a'(A^{-1})\,\left[ A^{-1}\right]'\right)\big|_{w}w,\quad w:=\lambda^{-1}|G|.
\]
Furthermore, we find that
\[
(A^{-1}(w))'=\frac{1}{(a(s)s)'}=\frac{1}{a(s)} \frac{1}{1+\frac{a'(s)s}{a(s)}},\qquad \frac{(A^{-1}(w))'w}{(A^{-1}(w))}=\frac{1}{1+\frac{a'(s)s}{a(s)}}.
\]
From this $B'=\frac{1}{2}D\frac{1}{D+1}$, where $D$ as above. Hence
\[
\left|\frac{B'}{1-B'}\right|=\left|\frac{D}{D+2}\right|\leq \frac{\min\{|\alpha-1|, |\beta-1|\}}{\alpha+1}.
\]
This is all that we need, because now \eqref{eq-22cmplx} reads $G_{\bz}+\overline{G_{\bz}}=0$ and $C=0$. Therefore, the counterpart of~\eqref{eq-final-cmplx} reads:
\[
 G_{\bz}-a_1G_z-a_2\overline{F_z}=-a_1\overline{G}\frac{\lambda_{z}}{\lambda}-a_2G\frac{\lambda_{\bz}}{\lambda}.
\]
with $a_1=-\frac{B'}{2(1-B')}$ and $a_2=-\overline{a_1}$. One directly checks that the ellipticity condition $|\frac{B'}{1-B'}|<1$ for $a=s^{p-2}$ reads $|\frac{B'}{1-B'}|=\frac{|p-2|}{p}<1$ which is exactly the formula after (2.9) on pg. 6 in ~\cite{strzelecki} for $a=p-2$. 
	\end{rem}

Notice, that on the contrary to the planar case (i.e. $\lambda=const$), we now cannot conclude that $F$ is a quasiregular map. Nevertheless, by the representation theorem on pg. 259 in \cite{bjs} we may write 
\begin{equation}\label{F-sol}
F(z)=\phi(w(z))\exp{\psi},
\end{equation}
where $\phi$ is holomorphic, $w$ is a H\"older continuous homeomorphism in $W^{1,2+\epsilon}_{loc}$  for some $\epsilon>0$ and $\psi$ is H\"older continuous. The Sobolev regularity of $w$ follows from the Gehring's lemma on higher integrability of quasiconformal mappings, see the proof of the representation formula on pg. 260 in \cite{bjs}. In particular, zeros of $F$ are governed by zeros of $\phi$, and hence by complex analysis critical points of $a$-harmonic functions $u$ are isolated and form a discrete set.

Moreover, by the discussion in~\cite[Chapter 6.4]{bjs}, the representation formula~\eqref{F-sol} implies the unique continuation property on smooth surfaces for $a$-harmonic equations in subject, in particular for the $p$-harmonic equation. This result is well known in the plane. 
In the $p$-harmonic case, the Riemannian counterpart is a direct consequence of~\cite{strzelecki}. However, according to our best knowledge, the general Riemannian result has not been observed in the literature so far, and therefore, we formulate it here below.

\begin{prop}\label{prop:uni-cont}
 Let $\Om\subset M$ be a bounded connected open set and assume that $\partial \Om$ satisfies the interior and exterior ball condition. Let further $\phi\in C^2(\overline{\Om})$ and $u$ be the unique weak solution of the following Dirichlet problem:
\begin{equation*}
 \begin{cases}
 \div(a(|\nabla u|_g)\nabla u)=0 & \hbox{in } \Om,\\
 u|_{\partial \Om}=\phi.
 \end{cases}
\end{equation*}
Then, $u$ satisfies the unique continuation property, provided that function $a$ satisfies conditions~\eqref{a-harm-ass1}, \eqref{a-harm-ass2} and~\eqref{a-harm-ass3}.
\end{prop}
Notice that if $u$ in the above proposition is a priori $C^2$, then we only need condition~\eqref{a-harm-ass1} to hold.

Finally, we are in a position to formulate and prove a key observation allowing us to study the isoperimetric inequality, namely that the gradient of an $a$-harmonic function in subject does not vanish. Therefore, we  formulate this observation as a separate result, see Lemma~\ref{non-zero-lemma} below. This lemma generalizes similar observation for harmonic functions on surfaces, see~\cite[Lemma 2.9]{adve}, for $a$-harmonic functions in the plane, see~\cite[Theorem 2.1]{al2} and also~\cite[Section 2]{adve} for further references. The proof of Lemma~\ref{non-zero-lemma}  is strictly following its harmonic counterpart in~\cite[Lemma 2.9]{adve}. Nevertheless, for the convenience of readers we recall the full proof, addressing the $a$-harmonic modifications.

\begin{lem}\label{non-zero-lemma}
Let $\Omega$ be a $C^{1,\alpha}$-topological annular domain in a $2$-di\-men\-sional Riemannian manifold $(M^2, g)$.
Let $t_1,t_2\in \R$ be such that $t_1<t_2$ and let us consider a continuous up to the boundary $a$-harmonic solution $u$ of the Dirichlet problem~\eqref{DP} in $\Om$ under assumptions~\eqref{a-harm-ass1}, \eqref{a-harm-ass2} and \eqref{a-harm-ass3}.
Then, it holds that $\nabla u\not=0$ on $\Omega$.
\end{lem}

\begin{proof}
 In order to show the assertion of the lemma let us suppose on the contrary that there exists $x_0\in \Om$ such that $\nabla u(x_0)=0$. Consider the corresponding level curve $\gamma=\{x\in \Om: u(x)=u(x_0)\}$. 
 
 \emph{Claim:} There exists at least two simple closed curves $\gamma'_i\subset \gamma$, $i=1,2$.
 
 \emph{Proof of the claim:} 
We introduce isothermal coordinates $(x,y)$ induced by a conformal chart $\phi:\Omega\to\R^2$. The existence of such global isothermal coordinate systems on annular domain can be justified as in \cite[Lemma 3.6]{adve}. Since the $a$-harmonic equation ~\eqref{def-weak-aharm} in such coordinates has the form~\eqref{def-weak-plane} we associate with it the first order complex equation~\eqref{eq-final-cmplx} satisfied by the complex function $F$ defined in~\eqref{e:F}. Under assumptions~\eqref{a-harm-ass1}, \eqref{a-harm-ass2} and \eqref{a-harm-ass3}, $F$ is in $W^{1,2}_{loc}$ due to Proposition \ref{lem-reg}, while~\eqref{eq-final-cmplx} turns out to be uniformly elliptic  under condition~\eqref{a-harm-ass1}. Since $\phi(x_0)$ remains a critical point for $u\circ \phi^{-1}$, we see  by the form of solution~\eqref{F-sol} and by the theory of planar holomorphic functions that the level curve in the neighbourhood of $\phi(x_0)$ forms a finite family consisting of at least two arcs intersecting at $\phi(x_0)$.
  Thus, for points on the level curve $\gamma$ we obtain, via $\phi^{-1}$, that there are at least two curves passing through $x_0$ contained in $\gamma$. If any of those branches would intersect $\partial \Om$, then by the assumption of continuity of $u$ up to the boundary, it would hold that $u(x_0)=t_1$ (or $u(x_0)=t_2$), hence the maximum of $u$ (or, respectively, minimum of $u$) would be attained in the interior of $\Om$, forcing $u=const$ by the strong maximum (respectively, minimum) principle, see~\cite[Theorem 8.5.1]{puse-book}, whose assumption (A2) is implied by our assumptions~\eqref{a-harm-ass1} and~\eqref{a-harm-ass2} (cf. the paragraph in Introduction following formulations of those two assumptions). This is impossible, since $t_1\not =t_2$.
 
 Next, we rule out the possibility that the level curve $\gamma$ terminates at a point inside $\Om$. Indeed, suppose that there exists $y_0\in \gamma \cap \Om$, where $\gamma$ terminates, and consider two cases.

If $\nabla u(y_0)\not=0$, then the implicit function theorem implies that $\gamma$ can not terminate inside $\Om$, since it must be at least $C^1$ in a neighbourhood of $y_0$.

If $\nabla u(y_0)=0$, then by the discussion above, $\gamma$ would branch at $y_0$, contradicting assumption that it terminates there. 

To summarize, since $\gamma$ does not intersect $\partial \Om$ and does not terminate in $\Om$, it must contain at least two simple closed curves, denoted $\gamma'_i$, $i=1,2$, obtained by gluing regular curves (contained in $\gamma$). This ends the proof of the claim.

Since none of the curves $\gamma'_1$ and $\gamma'_2$ touches the boundary of $\Omega$, a topological argument together with the maximum principle allow us to infer that at least one of them bounds a domain $\Om'\Subset \Om$. 
Therefore, $u$ is constant in $\Om'$, and thus by the unique continuation property, cf. Proposition~\ref{prop:uni-cont}, $u$ is constant in $\Om$, contradicting that $t_1<t_2$. Hence, we conclude that $\nabla u\not=0$ in $\Om$. 
\end{proof}

\begin{ex}
 We present a class of Riccati-type equations which are covered by the above discussion and, hence, have their sets of critical points isolated and discrete.
 
Let us consider the solution on the surface $M$ of the following equation with function $a$  as above and a real-valued function $b$ defined on a domain $\Om\subset M$ such that $b\in L^{\infty}(\Om)$:
\[
 \div(a(|\nabla u|_g)\nabla u)=b(x)|\nabla u|_g^q,\quad 1\leq q<\infty.
\]
Observe that \eqref{eq-11cmplx} holds for this equation as well, since it corresponds to equality of mixed second order derivatives of $u$ in the distributional sense. Furthermore, by direct computations we obtain that \eqref{eq-22cmplx} holds with addtional expression
\[
 \gamma=\frac{2b(z) [A^{-1}(\lambda^{-\frac12}|F|]^q)}{3A^{-1}(\lambda^{-\frac12}|F|)\lambda^{-\frac12}|F|-(A^{-1})'(\lambda^{-\frac12}|F|)\lambda^{-1}|F|^2}.
\]
In a consequence~\eqref{eq-final-cmplx} takes the following modified form 

\begin{equation*}
 F_{\bz}-a_1F_z-a_2\overline{F_z}=-a_1\overline{F}\frac{\lambda_{z}}{\lambda}-a_2F\frac{\lambda_{\bz}}{\lambda}+\frac{\gamma}{2}.
\end{equation*}

In order to apply the representation theorem on pg. 259 in~\cite{bjs} we need to have bounded $\gamma=\gamma(z)$, that is independent of $F$, cf. formulas (8) and (9) on pg. 257 in~\cite{bjs}. If $a(s)=s^{p-2}$, then by computations similar to the one in Example 3 above, we have that for $A(s)=a^{\frac12}(s)s$ the following holds
\begin{align*}
 \gamma&=2b(z)\frac{[A^{-1}(\lambda^{-\frac12}|F|]^{q-1})}{\lambda^{-\frac12}|F|\left(3-\frac{(A^{-1})'(\lambda^{-\frac12}|F|)\lambda^{-\frac12}|F|}{A^{-1}(\lambda^{-\frac12}|F|)}\right)}=  2b(z)\frac{s^{q-1}}{a^{\frac12} (\lambda^{-\frac12}|f|)\lambda^{-\frac12}|f|\left(3-\frac{1}{\frac{1}{2}\frac{a'(s)s}{a(s)}+1}\right)} \\
 &=\frac{2p}{3p-2}b(z)s^{q-1-\frac{p}{2}}=\frac{2p}{3p-2}b(z),
\end{align*}
provided that $q=1+\frac{p}{2}$. Since $b$ is bounded, then so is $\gamma$ and we can apply the above discussion also to equations
\[
\div(|\nabla u|_g^{p-2}\nabla u)=b(x)|\nabla u|^{1+\frac{p}{2}}_g.
\]
The above computations open possibility to establish the no-critical points lemma in the setting of Riccati-type equations on surfaces. Moreover, the isoperimetric inequalities for such equations can also be investigated upon establishing formulas for $L'$ and $L''$. However, we leave this task to the future projects.
\end{ex}

\section*{Appendix A}

In this Appendix we adapt to our setting the regularity theory due to Talenti \cite{ta}, see also \cite{mawe}. Unlike the original approach, we restrict to the two dimensional case, which permits us to write a quite short and almost self-contained proof.  

\begin{theorem}\label{th:regularity}
	Let $v\in C^2(\Omega)$ be a solution to 
	\[
	\div(a(|\nabla v|_g)\nabla u)=0
	\]	
	in a domain $\Omega\subset M^2$, where $a\in C^1(0,\infty)$  is a positive function satisfying \eqref{a-harm-ass1} and $\inf a(s)>0$. 
 Then, for all $C^0$-smooth 
domains $U\Subset V\Subset \Omega$ it holds
	\[
	\|v\|_{W^{2,2}(U)}\le C,
	\]
	for some constant $C>0$ which depends on $U,V,\|v\|_{W^{1,2}(V)}$, constants $\alpha,\beta$ in condition~\eqref{a-harm-ass1} and $\|\lambda\|_{C^1(V)}$, the norm of the conformal factor given by the isothermal coordinates on $M^2$.
\end{theorem}

We remark that condition ensuring $\inf a(s)>0$ states that the equation in subject is nondegenerate elliptic.

Set $\mathcal A (s)=\frac{a'(s)s}{a(s)}$ and introduce the operator $L$ acting on $C^2$ functions defined by
\[
Lv=\Delta v + \mathcal A(\lambda^{-1}|\nabla v|_0)\frac{\nabla v \nabla^2 v (\nabla v)^T}{|\nabla v|_0^2}=\sum_{i,j=1}^2a_{ij}(v)v_{ij},
\]
where $a_{ij}(v)=\delta_{ij}+\mathcal A(\lambda^{-1}|\nabla v|_0)\frac{v_iv_j}{v_1^2+v_2^2}$. Then,  $v$ satisfies
\[
Lv=\cA(\lambda^{-1}|\nabla v|_0)
\langle \frac{\nabla \lambda}{\lambda}, \nabla v \rangle_0.
\]
We need the following result (see \cite[Theorem 2]{ta}).

\begin{prop}\label{prop:talenti}
	Let $w\in C^2(V)$ be a function such that $\supp(w)\Subset V$. Then
	\[
	\int_V \sum_{i,j=1}^2w_{ij}^2\le c\int_V (\sum_{i,j=1}^2a_{ij}(v)w_{ij})^2,
	\]
	for a constant $c>0$ which depends on $\alpha,\beta$.
\end{prop}

\begin{proof}
Without loss of generality, by approximation we can assume that $w\in C^3(V)$. Fix $x\in V$. Consider the real symmetric $2\times 2$ matrix $\mathfrak a = (a_{ij}(v))$ at $x$. Note that $\tr\fra=2+\cA$ and $\tr\fra^2=\sum_{i,j=1}^2a_{ij}^2=2+2\cA+\cA^2$. As $x$ and $w$ are fixed, in what follows we omit the argument of $\cA=\cA(\lambda^{-1}(x)|\nabla w|_0)$. Choose constants $c_1,c_2$ such that
	\begin{align}\label{e:cordes}
	c_1&>\frac{1+\beta^2}{2\alpha}\geq\frac{(\cA+1)^2+1}{2(1+\cA)}=\frac{2+2\cA+\cA^2}{2+2\cA}=\frac{\tr\fra^2}{(\tr\fra)^2-\tr\fra^2}\\
	c_2&=\frac{c_1^2-1}{2c_1\alpha-1-\beta^2}\ge\frac{c_1^2-1}{c_1((\tr\fra)^2-\tr\fra^2)-\tr\fra^2}.
	\nonumber\end{align}
The last inequality is due to the fact that denumerator of the right-hand side can be written as $2c_1(1+\cA)-1-(1+\cA)^2$ and it is thus positive when $1+\cA\in[\alpha,\beta]$ due to the choice of $c_1$.
We observe that the first inequality \eqref{e:cordes} corresponds to what is known in the literature as Cordes' condition, cf. formula (1.7) in~\cite{mawe}. Indeed, in our notation and for $n=2$ the Cordes' condition reads:
\[
\tr\fra^2 \leq (1+\delta)^{-1} (\tr\fra)^2
\]
for some $\delta\in (0,1]$, that is,
\[
\frac{1}{\de}\ge \frac{\tr\fra^2}{(\tr\fra)^2-\tr\fra^2}.
\]
This latter is verified by the choice $\delta=c_1^{-1}<\frac{2\alpha}{1+\beta^2}$.

Next, we need the following algebraic observation.
\medskip

\noindent
{\bf Claim:}
\begin{equation}\label{e:alg lemma}
\sum_{i,j=1}^2w_{ij}^2+2c_1\det\Hess\, w\le c_2 (\sum_{i,j=1}^2a_{ij}(v)w_{ij})^2.
\end{equation}

\emph{Proof: } Since $\frp:=\Hess\, w$ is a real symmetric matrix, it holds that $\frp=\frm^{-1}\frk\frm$ for some orthogonal matrix $\frm$ and some diagonal matrix $\frk=\left(\begin{array}{cc}k_1&0\\0&k_2\end{array}\right)$. Define $\frb=\frm^{-1}\fra\frm$. We have
$\sum_{i,j=1}^2w_{ij}^2=\tr \frp^2 =\tr \frk^2$,  $\det\Hess\, w=\det \frp = \det \frk$ and 
\[
\sum_{i,j=1}^{2} a_{ij}(v)w_{ij}=\tr(\fra\frp)=\tr(\frb\frk)=b_{11}k_1+b_{22}k_2.
\]
Hence \eqref{e:alg lemma} is implied by 
\[
k_1^2+k_2^2+2c_1k_1k_2\le c_2(b_{11}k_1+b_{22}k_2)^2,\quad\hbox{for all } k_1,k_2\in\R.
\]
Up to rescaling $(k_1,k_2)$ it is enough to prove that
\[ k_1^2+k_2^2+2c_1k_1k_2\le c_2,\quad\hbox{for all } k_1,k_2\in\R\ : \ b_{11}k_1+b_{22}k_2=1.
\]
Note that 
$k_1^2+k_2^2+2c_1k_1k_2= c_2$ is a hyperbola whose symmetry axis is the line $k_1=k_2$. Accordingly, it is enough to prove that this hyperbola does not intersect the line $b_{11}k_1+b_{22}k_2=1$, i.e. that the system 
\[
\begin{cases}
	k_1^2+k_2^2+2c_1k_1k_2= c_2\\b_{11}k_1+b_{22}k_2=1
\end{cases}
\] admits no solutions. Upon computing $k_1$ from the second equation and substituting it in the first one, yields the following second order equation in $k_2$
\[
(b_{22}^2+b_{11}^2-2c_1b_{11}b_{22})k_2^2+2(c_1b_{11}-b_{22})k_2 + (1-c_2b_{11}^2)=0
\]
whose discriminant $\Delta$ satisfies  
\begin{align*}
\frac{\Delta}{4b_{11}^2} &=b_{11}^{-2} \left[(c_1b_{11}-b_{22})^2- (b_{22}^2+b_{11}^2-2c_1b_{11}b_{22})(1-c_2b_{11}^2)\right]\\
&=c_1^2-1-c_2(2c_1b_{11}b_{22}-b_{11}^2-b_{22}^2)\\
&=c_1^2-1-c_2\left[c_1((b_{11}+b_{22})^2-(b_{11}^2+b^2_{22}))-(b_{11}^2+b^2_{22})\right]\\
&<c_1^2-1-c_2[c_1((\tr\fra)^2-\tr\fra^2)-\tr\fra^2]<0,
\end{align*}
as $\tr\fra = \tr\frb$ and $\tr\fra^2=\tr\frb^2=\sum_{i,j=1}^2b_{ij}^2\ge b_{11}^2+b_{22}^2$. Thus, the proof of the claim is complete.

\smallskip

An explicit computation shows that for any $C^3$ function $w$,
\[
\det \Hess\,w = \frac{1}{2}\div(\Delta w \nabla w	) -\frac{1}{4}\Delta|\nabla w|^2.
\]
By the assumptions of Proposition~\ref{prop:talenti}, $w$ is zero in a neighborhood of $\partial V$, and so the Stokes theorem implies $\int_V \det \Hess\,w = 0$.
 Therefore,  we conclude the proof of Proposition \ref{prop:talenti} by integrating \eqref{e:alg lemma} over $V$. 
\end{proof} 

Once we have Proposition \ref{prop:talenti}, the proof ot Theorem \ref{th:regularity} can be done by mimicking the proof of Theorem 9.11 in \cite{gt}.

\begin{proof}[Proof of Theorem \ref{th:regularity}]
Let $\vp\in C^\infty_c(\Omega)$ such that $\vp\equiv 1$ on $U$ and $\supp\,\vp\Subset V$. Furthermore, let $|\nabla \phi|\leq \frac{c}{\dist(\partial U, \partial V)}$ and let similar growth condition hold for $|\nabla^2 \phi|$.
 By Proposition \ref{prop:talenti} applied to $\vp v$, we obtain that
	\begin{equation}\label{e:integral Talenti}
\int_{U} \sum_{i,j=1}^2 v_{ij}^2\le\int_{V} \sum_{i,j=1}^2(\vp v)_{ij}^2\le 
c\int_V \left(\sum_{i,j=1}^2a_{ij}(v)(\vp v)_{ij}\right)^2.
\end{equation}
We compute
\begin{align}\label{e:leibnitz}
\left(\sum_{i,j=1}^2a_{ij}(v)(\vp v)_{ij}\right)^2
&=\left(\sum_{i,j=1}^2a_{ij}(v)\vp_{ij} v + 2\sum_{i,j=1}^2a_{ij}(v)\vp_i v_{j} +  \vp Lv\right)^2\\
&\le 3\sum_{ij}\left(\|a_{ij}(v)\|^2_{L^\infty(V)}\|\vp_{ij}\|^2_{L^\infty(V)}  v^2 +\|a_{ij}(v)\|^2_{L^\infty(V)}\|\vp_i\|^2_{L^\infty(V)}  v_j^2\right) \nonumber\\
&+ \frac34\|\vp\|_{L^\infty(V)}^2 \|\cA(\lambda^{-\frac12}|\nabla v|_0)\|_{L^\infty(V)}^2
\| \nabla\log \lambda\|_{L^\infty(V)}|\nabla v|^2. \nonumber
\end{align}
Since $\|a_{ij}(v)\|_{L^\infty(V)}$ and $\|\cA(\lambda^{-\frac12}|\nabla v|_0)\|_{L^\infty(V)}$ can be upper bounded in terms of $\alpha$ and $\beta$ and $\|\vp\|_{C^2(V)}$ can be estimated in terms of $U$ and $V$, inserting \eqref{e:leibnitz} into \eqref{e:integral Talenti} concludes the proof.  

\end{proof}

\end{document}